\documentclass[12pt]{amsart}
\usepackage[T1]{fontenc}
\usepackage[latin9]{inputenc}
\usepackage{geometry}
\usepackage{mathrsfs}
\usepackage{amstext}
\usepackage{amsthm}
\usepackage{amssymb}

\usepackage{comment}

\makeatletter
\numberwithin{equation}{section}
\numberwithin{figure}{section}
\theoremstyle{plain}
\newtheorem{thm}{\protect\theoremname}[section]
\theoremstyle{definition}

\theoremstyle{definition}

\theoremstyle{plain}

\theoremstyle{plain}
\newtheorem{lem}[thm]{\protect\lemmaname}
\theoremstyle{plain}

\theoremstyle{definition}
\newtheorem{exa}[thm]{\protect\examplename}
\theoremstyle{definition}

\newtheorem*{def*}{Definition}
\newtheorem*{rem*}{Remark}
\newtheorem*{exa*}{Example}

\renewenvironment{proof}[1][\proofname]{\medskip \noindent {\bfseries #1. }}{\hfill \qedsymbol\medskip}

\usepackage[svgnames]{xcolor}
\usepackage[bookmarksnumbered=true]{hyperref} 
\hypersetup{
	colorlinks = true,
	linkcolor = Blue,
	anchorcolor = blue,
	citecolor = Green,
	filecolor = blue,
	urlcolor = FireBrick
}
\usepackage{bbm}

\MakeRobust{\ref}
\newcommand{\labeltext}[2]{
\@bsphack
\csname phantomsection\endcsname
\def\@currentlabel{#1}{\label{#2}}
\@esphack
}

\usepackage{scalerel}
\usepackage[usestackEOL]{stackengine}
\def\dashint{\,\ThisStyle{\ensurestackMath{%
  \stackinset{c}{.2\LMpt}{c}{.5\LMpt}{\SavedStyle-}{\SavedStyle\phantom{\int}}}%
  \setbox0=\hbox{$\SavedStyle\int\,$}\kern-\wd0}\int}

\usepackage{enumitem}

\DeclareRobustCommand{\SkipTocEntry}[5]{}

\newcommand{\mR}{\mathbb{R}}   
\newcommand{\mN}{\mathbb{N}}   
\newcommand{\mK}{\mathbb{K}}   
\newcommand{\mS}{\mathbb{S}}   

\newcommand{\abs}[1]{\lvert #1 \rvert}  
\newcommand{\norm}[1]{\lVert #1 \rVert}  
\newcommand{\br}[1]{\langle #1 \rangle}  

\newcommand{\mD}{\mathscr{D}}
\newcommand{\mE}{\mathscr{E}}

\newcommand{\mF}{\mathscr{F}}

\newcommand{\mJ}{\mathcal{J}}

\newcommand{\calE}{\mathcal{E}}
\newcommand{\calS}{\mathcal{S}}

\newcommand{\bfk}{\mathbf{k}}
\newcommand{\bfu}{\mathbf{u}}
\newcommand{\bfv}{\mathbf{v}}
\newcommand{\bfw}{\mathbf{w}}
\newcommand{\bfz}{\mathbf{z}}

\newcommand{\bfH}{\mathbf{H}}

\newcommand{\sfm}{\mathsf{m}}

\DeclareMathOperator{\Bal}{Bal}

\DeclareMathOperator{\supp}{supp}

\DeclareMathOperator{\curl}{curl}

\newcommand{\rmd}{\mathrm{d}}

\makeatother

\usepackage{babel}
\providecommand{\corollaryname}{Corollary}
\providecommand{\definitionname}{Definition}
\providecommand{\lemmaname}{Lemma}
\providecommand{\propositionname}{Proposition}
\providecommand{\remarkname}{Remark}
\providecommand{\theoremname}{Theorem}
\providecommand{\examplename}{Example}
\providecommand{\assumptionname}{Assumption}

\thanks{P.-Z. Kow was supported by the NCCU Office of research and development and the National Science and Technology Council of Taiwan, NSTC 112-2115-M-004-004-MY3; H. Shahgholian was supported by Swedish Research Council grant nr 2021-03700.} 
\thanks{We would like to express our sincere gratitude to the anonymous referee for their meticulous review of the manuscript and for graciously pointing out several inaccuracies and partial flaws.}

\begin{document}

\title[Multi-phase $k$-quadrature domains and applications]{Multi-phase $k$-quadrature domains and applications to acoustic waves and magnetic fields}

\begin{sloppypar}

\begin{abstract}
The primary objective of this paper is to explore the multi-phase variant of quadrature domains associated with the Helmholtz equation, commonly referred to as $k$-quadrature domains. Our investigation employs both the minimization problem approach, which delves into the segregation ground state of an energy functional, and the partial balayage procedure, drawing inspiration from the recent work by Gardiner and Sj\"odin. Furthermore, we present practical applications of these concepts in the realms of acoustic waves and magnetic fields.

\end{abstract}

\author{Pu-Zhao Kow}
\address{Department of Mathematical Sciences, National Chengchi University, No. 64, Sec. 2, ZhiNan Rd., Wenshan District, Taipei 116302, Taiwan.}
\email{pzkow@g.nccu.edu.tw}

\author{Henrik Shahgholian}
\address{Department of Mathematics, KTH Royal Institute of Technology, SE-10044 Stockholm, Sweden. }
\email{henriksh@kth.se}

\maketitle

\tableofcontents{}

\section{Introduction}

\subsection{Background}
The subject under consideration in this note is $k$-quadrature domains with $k > 0$, also known as quadrature domains for the Helmholtz operator. This topic is intricately connected to the inverse scattering theory, as detailed in references \cite{SS21NonscatteringFreeBoundary, KLSS22QuadratureDomain, KSS23Minimization}.

Given any $\mu \in \mE'(\mR^{n})$ ($n\geq 2$), we recall that a bounded open set $D \subset \mR^{n}$ is termed a (one-phase) $k$-quadrature domain with respect to $\mu$ if $\mu \in \mE'(D)$ and
\begin{equation*}
\int_{D} w(x) \ \rmd x = \br{\mu,w}
\end{equation*}
holds for all $w \in L^{1}(D)$ satisfying $(\Delta + k^{2})w=0$ in $D$. It is shown  in \cite[Proposition~2.1]{KLSS22QuadratureDomain} that a bounded open set $D \subset \mR^{n}$ is a $k$-quadrature domain for $\mu \in \mE'(D)$ if and only if there exists $\tilde{u}\in\mD'(\mR^{n})$ satisfying the following equations
\begin{equation}
\left\{\begin{aligned}
& (\Delta + k^{2})\tilde{u} = \chi_{D} - \mu && \text{in $\mR^{n}$,} \\
& \tilde{u} = \abs{\nabla \tilde{u}}=0 && \text{in $\mR^{n}\setminus D$.}
\end{aligned}\right. \label{eq:k-QD-1}
\end{equation}

Later in \cite{KSS23Minimization}, the system of equations \eqref{eq:k-QD-1} was further generalized by introducing the following Bernoulli-type free boundary problem:
\begin{equation}
\left\{
\begin{aligned}
& (\Delta + k^{2})\tilde{u} = h - \mu && \text{in $D$,} \
& \tilde{u} = 0 && \text{in $\mR^{n}\setminus \partial D$,} \
& \abs{\nabla\tilde{u}} = g && \text{in $\mR^{n}\setminus \partial D$,}
\end{aligned}\right. \label{eq:k-hybrid-QD-2}
\end{equation}
where the Bernoulli condition $\abs{\nabla\tilde{u}} = g$ is considered in a very weak sense. Refer also to  \cite{KSS23Anisotropic} for a connection between the anisotropic non-scattering problem and the Bernoulli-type free boundary problem. We refer to a bounded domain $D$ in \eqref{eq:k-hybrid-QD-2} as the hybrid $k$-quadrature domain.

\subsection{Two- and multiphase \texorpdfstring{$k$}{k}-quadrature domains (the notion)}

A bounded domain $D$ in $\mR^{n}$ is referred to as a quadrature domain for harmonic functions, associated with a distribution $\mu \in \mE'(D)$ if
\begin{equation}
\int_{D} h(x) \  \rmd x = \int h(x) \  \rmd \mu(x) \label{eq:quadrature-identity-classical}
\end{equation}
holds for every harmonic function $h \in L^{1}(D)$; see, for example, the monograph \cite{Sak82QuadratureDomains}. In the special case when $\mu = \sum_{j=1}^{m} \lambda_{j}\delta_{x_{j}}$, where $\delta_{a}$ is the Dirac measure at $a$, \eqref{eq:quadrature-identity-classical} reduces to a quadrature identity for computing integrals of harmonic functions; refer to \cite{GS05QuadratureDomain}. Quadrature domains can also be regarded as a generalization of the mean value theorem for harmonic functions: $B_{r}(a)$ is a quadrature domain with $\mu = \abs{B_{r}(a)}\delta_{a}$. Various examples can be constructed using complex analysis; see, for instance, \cite{Dav74SchwartzFunction,GS05QuadratureDomain,Sak83BalayageQuadrature} for further background.

A generalization of the Helmholtz operator was investigated in \cite{KLSS22QuadratureDomain,GS24PartialBalayageHelmholtz}. For $k > 0$, a bounded open set $D$ in $\mR^{n}$ (not necessarily connected) is referred to as a quadrature domain for $(\Delta + k^{2})$, or a $k$-quadrature domain, associated with a distribution $\mu \in \mE'(D)$, if
\begin{equation}
\int_{D} w(x) \  \rmd x = \int w(x) \  \rmd \mu(x) \label{eq:k-QD-def}
\end{equation}
holds for all $w \in L^{1}(D)$ satisfying $(\Delta + k^{2})w=0$ in $D$. It is essential to note that the $k$-quadrature domain can also be regarded as a generalization of the mean value theorem: For each $k > 0$, $B_{r}(a)$ is a $k$-quadrature domain with $\mu = c_{k,k,r}^{\rm MVT} \delta_{a}$ for some suitable constant $c_{n,k,r}^{\rm MVT}$ (which may be zero for specific parameters $n$, $k$, and $r$). Similar to the classical case ($k=0$), various examples can also be constructed using complex analysis \cite{KLSS22QuadratureDomain}.

We now introduce the concept of two-phase quadrature domains as defined in \cite{EPS11TwoPhaseQD}, and later in \cite{GS12TwoPhaseQD}.

Let $D_{\pm}$ be disjoint bounded open subsets of $\mR^{n}$, and let $\mu_{\pm} \in \mE'(D_{\pm})$, respectively. If a pair $(D_{+}, D_{-})$ has the property that
\begin{equation}
\int_{D_{+}} h(x) \  \rmd x - \int_{D_{-}} h(x) \  \rmd x = \int h(x) \  \rmd (\mu_{+} - \mu_{-}) \label{eq:quadrature-identity-classical-2-phase}
\end{equation}
holds for every harmonic function $h$ on $D_{+}\cup D_{-}$ with $h \in C(\overline{D_{+}\cup D_{-}})$, then we designate such a pair $(D_{+}, D_{-})$ as a \emph{two-phase quadrature domain} (for harmonic functions) corresponding to distributions $(\mu_{+}, \mu_{-}) \in \mE'(D_{+})\times \mE'(D_{-})$. The precise meaning in the right hand side of \eqref{eq:quadrature-identity-classical-2-phase} is the distributional pairing $\br{\mu_{+},h} - \br{\mu_{-},h}$, which is well-defined since $h \in C^{\infty}(D_{\pm})$. The following trivial example also illustrates this notion:

\begin{exa*} 
It is evident that if $D_{\pm}$ are quadrature domains (for harmonic functions) corresponding to distributions $\mu\in\mE'(D_{\pm})$ respectively in the sense of \eqref{eq:quadrature-identity-classical} and satisfy $\overline{D_{+}} \cap \overline{D_{-}} = \emptyset$, then such a pair $(D_{+},D_{-})$ clearly satisfies \eqref{eq:quadrature-identity-classical-2-phase}. 
\end{exa*}
Here, we refer to \cite{EPS11TwoPhaseQD,GS12TwoPhaseQD} for some less trivial examples of two-phase quadrature domains \eqref{eq:quadrature-identity-classical-2-phase}. It is worth mentioning that Gardiner and Sj\"odin \cite[Theorem~3.1(b)]{GS12TwoPhaseQD} proved that if $(D_{+},D_{-})$ is a two-phase quadrature domain (for harmonic functions) corresponding to distributions $(\mu_{+},\mu_{-}) \in \mE'(D_{+})\times \mE'(D_{-})$, then there exist polar sets $Z_{\pm}$ such that there exists $\tilde{u}$ such that
\begin{equation*}
\Delta \tilde{u} = (1 - \mu_{+})\chi_{D_{+}\cup Z_{+}} - (1 - \mu_{-})\chi_{D_{-}\cup Z_{-}} \text{ in $\mR^{n}$} ,\quad \tilde{u} = 0 \text{ outside $D_{+}\cup Z_{+} \cup D_{-} \cup Z_{-}$.}
\end{equation*}
Conversely, if a function  $\tilde{u} \in H^{1}(\mR^{n})$ with compact support satisfies\footnote{We define the sets $\{v>0\} := \mR^{n}\setminus \{v \le 0\}$ and $\{v<0\} := \mR^{n}\setminus \{v \ge 0\}$. In the case when $v \in H^{1}(\mR^{n})$ \emph{has compact support}, the inequalities $v \ge 0$ and $v \le 0$ also can be understood in \emph{almost-every pointwise} (a.e.) sense, see \cite[Definition~II.5.1 and Proposition~5.2]{KS00IntroductionVariationalInequalities}.} 
\begin{equation*}
\Delta \tilde{u} = (1 - \mu_{+})\chi_{\{\tilde{u}>0\}} - (1 - \mu_{-})\chi_{\{\tilde{u}<0\}} \text{ in $\mR^{n}$,}
\end{equation*}
and if $\supp\,(\mu_{\pm}) \subset \{\pm\tilde{u}>0\}$, by \cite[Theorem~3.1(a)]{GS12TwoPhaseQD} we see that $\left(\{\tilde{u}>0\},\{\tilde{u}<0\}\right)$ a two-phase quadrature domain (for harmonic functions) corresponding to distributions $(\mu_{+},\mu_{-}) \in \mE'(D_{+})\times \mE'(D_{-})$. 

In \cite{EPS11TwoPhaseQD}, they use a minimization approach to study the model equation 
\begin{equation}
\Delta \tilde{u} = (\lambda_{+} - \mu_{+})\chi_{\{\tilde{u}>0\}} - (\lambda_{-} - \mu_{-})\chi_{\{\tilde{u}<0\}} \text{ in $\mR^{n}$.} \label{eq:classical-2-phase}
\end{equation}
The above discussions strongly suggest  studying  the following model equation, which is also the main theme of this paper: 

\begin{equation}
\Delta \tilde{u} + k_{+}^{2}\tilde{u}_{+} - k_{-}^{2}\tilde{u}_{-} = (\lambda_{+} - \mu_{+})\chi_{\{\tilde{u}>0\}} - (\lambda_{-} - \mu_{-})\chi_{\{\tilde{u}<0\}} \quad \text{in $\mR^{n}$,} \label{eq:main-2-phase}
\end{equation} 
where $k_{\pm}\ge 0$, $\lambda_{\pm} > 0$ and $\mu_{\pm} \in \mE'(\mR^{n})$. We refer to such pair of domains $(D_{+},D_{-})$ with $D_{\pm} = \{\pm \, \tilde{u}>0\}$ as  the \emph{two-phase $(k_{+},k_{-})$-quadrature domain}. 

In \cite{AS16MultiPhaseQD},  the following problem in terms of partial differential equations was considered:
Given $m$ positive measures $\mu_{i}$ and constants $\lambda_{i}$, for $i=1,\cdots,m$,  find functions $u_{i} \ge 0$ with suitable regularity and disjoint sets  $D_{i}=\{u_{i}>0\}$ such that 
\begin{equation}
\Delta(u_{i}-u_{j}) = (\lambda_{i}-\mu_{i})\chi_{D_{i}} - (\lambda_{j}-\mu_{j})\chi_{D_{j}} \text{ in $\mR^{n}\setminus \bigcup_{\ell\neq i,j} \overline{D_{\ell}}$.} \label{eq:classical-local-2-phase}
\end{equation}
It is easy to see that \eqref{eq:classical-2-phase} is simply a special case of \eqref{eq:classical-local-2-phase} for $m=2$. We remark that \eqref{eq:classical-local-2-phase} is locally a two-phase problem in the set $\mR^{n}\setminus \bigcup_{\ell\neq i,j} \overline{D_{\ell}}$, which excludes all points in $\partial D_{i}\cap \partial D_{j}$, therefore it is not easy to establish quadrature identities similar to \eqref{eq:quadrature-identity-classical-2-phase} for multi-phase case. One can think about the \emph{Lakes of Wada}, which are three disjoint connected open sets of the plane or open unit square with the counter-intuitive property that they all have the same boundary. Indeed, one can also construct a countable infinite number of disjoint connected open sets of the plane with the same boundary.

According to \cite{AS16MultiPhaseQD}, inspired by the segregation problem \cite{CTV05SegregationProblem}, they minimize some suitable energy functional so that the minimizer satisfies \eqref{eq:classical-local-2-phase}. In other words, the supports of densities $u_{i}$ have to satisfy a suitable optimal partition problem in $\mR^{n}$.

Not only the multi-phase problem which we considered, there are also some other type of segregation problem, for example \cite{CKL2009}:
they minimize the Dirichlet functional (say)
\begin{equation*}
\mathcal{D}_\epsilon (\bfu) := \sum_{i=1}^{m} \int_{\Omega} \abs{\nabla u_{i}(x)}^{2} \, \rmd x + f_\epsilon (u) 
\end{equation*}
where $f_\epsilon (u) $ is chosen such that the functional gets huge penalties, say $1/\epsilon$, 
on the set $\left\{u_i > 0\right\} \cap \left\{ u_j > 0\right\}$, and the limit of this functional leads to a segregation of the supports of the components.
The work \cite{AS16MultiPhaseQD} or the equation \eqref{eq:classical-local-2-phase} strongly suggests studying the following model equation: We want to find functions $u_{i} \ge 0$ with disjoint positivity sets $D_{i}=\{u_{i}>0\}$ such that
\begin{equation*}
\Delta (u_{i}-u_{j}) + k_{i}^{2} u_{i} - k_{j}^{2} u_{j} = (\lambda_{i}-\mu_{i})\chi_{D_{i}} - (\lambda_{j}-\mu_{j})\chi_{D_{j}} \text{ in $\Omega \setminus \bigcup_{\ell\neq i,j} \overline{D_{\ell}}$}
\end{equation*}
where $k_{i} \ge 0$, $\lambda_{i} \ge 0$ and $\mu_{i} \in \mE'(\mR^{n})$, for some open domain $\Omega$. We refer to such  $k$-tuple of domains $(D_{1},\cdots,D_{m})$ as the \emph{multi-phase $(k_{1},\cdots,k_{m})$-quadrature domain} (in $\Omega$).

\subsection{Applications: acoustic waves and magnetic fields} 

\subsubsection{Inverse scattering in acoustic waves}
To provide motivation for this study, we initially establish a connection between the two-phase problem \eqref{eq:main-2-phase} and the inverse scattering problem for acoustic waves. We consider the acoustic scattering problem governed by the wave equation $c(x)^{-2}\partial_{t}^{2}U - \Delta U=0$, where $c$ is the velocity of sound in the given medium. The acoustic wave with a fixed frequency (wave number) $k_{0}>0$ corresponds to solutions of the form $U(x,t)=e^{\mathbf{i}k_{0}t}u_{\tilde{\rho}}(x)$, where the total field $u_{\tilde{\rho}}$ satisfies the (inhomogeneous) Helmholtz equation
\begin{equation}
\Delta u_{\tilde{\rho}} + k_{0}^{2}\tilde{\rho}(x)u_{\tilde{\rho}}=0 \quad \text{in $\mR^{n}$} \label{eq:Helmholtz-model}
\end{equation}
where (for simplicity) we have set $\tilde{\rho} = \tilde{\rho} (x)  = c(x)^{-2} $. Later, we will also explain that \eqref{eq:Helmholtz-model} models the cylindrical magnetic field; refer to Section~\ref{sec:Magnetic-field} below.

If we expose the medium with an incoming wave $u_{0}$ that solves
\begin{equation}
(\Delta + k_{0}^{2})u_{0} = 0 \quad \text{in $\mR^{n}$,} \label{eq:incident-field-u0}
\end{equation}
then the total field $u_{\tilde{\rho}}$, which verifies \eqref{eq:Helmholtz-model}, has the form $u_{\tilde{\rho}} = u_{0}+u^{\rm sc}$ for some scattered wave $u^{\rm sc}$, which is outgoing. Classical scattering theory \cite{CCH23InverseScatteringTransmission,CK19scattering,KG08Factorization} guarantees the existence and uniqueness of such an outgoing scattered field $u^{\rm sc}\in H_{\rm loc}^{1}(\mR^{n})$.

In order to define non-scattering phenomenon, we need to recall some background in the topic. Indeed, we first recall that 
a solution $v$ of $(\Delta + k_{0}^{2})v=0$ in $\mR^{n}\setminus B_{R}$ $(\text{for some $R>0$})$ is outgoing if it satisfies the Sommerfeld radiation condition 
\begin{equation}\label{sommerfeld}
\lim_{\abs{x}\rightarrow\infty} \abs{x}^{\frac{n-1}{2}} (\partial_{\abs{x}}v - \mathbf{i}k_{0}v)=0 \quad \text{uniformly in all directions $\hat{x}=\frac{x}{\abs{x}} \in \mathcal{S}^{n-1}$,}
\end{equation}
where $\partial_{\abs{x}}=\hat{x}\cdot\nabla$ denotes the radial derivative. In this case, the far-field pattern $v^{\infty}$ of $v$ is defined by 
\begin{equation*}
v^{\infty}(\hat{x}) := \lim_{\abs{x}\rightarrow\infty} \gamma_{n,k_{0}}^{-1} \abs{x}^{\frac{n-1}{2}} e^{-\mathbf{i}k_{0}\abs{x}}v(x) \quad \text{for all $\hat{x}\in\mathcal{S}^{n-1}$}
\end{equation*}
for some normalizing constant $\gamma_{n,k_{0}}\neq 0$. The Rellich uniqueness theorem \cite{CK19scattering,Hormander_rellich} implies that $v^{\infty}\equiv 0$ if and only if $v=0$ in $\mR^{n}\setminus B_{R}$. 

To formulate our  theorem, regarding applications to acoustic waves, we need the following definition.

\begin{def*}[Non-scattering]
Consider two acoustic-penetrable obstacles (medium) $(D_{\pm},\rho_{\pm})$ such that $D_{+}\cap D_{-}=\emptyset$ with refraction indices (the light bending ability of that medium) $\rho_{\pm} \in L^{\infty}(D_{\pm})$.
We call $\rho_{\pm} \in L^{\infty}(D_{\pm})$ \emph{signed contrasts} if $\rho_{\pm} \ge c > 0$ near $\partial D_{\pm}$, respectively. For each fixed wave number $k_{0}>0$, we illuminate the obstacles $(D_{\pm},\rho_{\pm})$ using the incident field $u_{0}$ as in \eqref{eq:incident-field-u0}, producing a unique total field $u_{\rho_{\pm}} = u_{0} + u_{\rm sc}$ (see \eqref{eq:Helmholtz-model}) satisfying
\begin{equation}
(\Delta + k_{0}^{2} + \rho_{+}\chi_{D_{+}} - \rho_{-}\chi_{D_{-}})u_{\rho_{\pm}} = 0 \quad \text{in $\mR^{n}$.} \label{eq:total-field-u-rho}
\end{equation}
We say that the pair of obstacles $(D_{\pm},\rho_{\pm})$ is non-scattering with respect to the incident field $u_{0}$ if $u_{\rm sc}=0$ outside $B_{R}$ for some sufficiently large $R>0$.
\end{def*}

\begin{rem*}
If $(D_{\pm},\rho_{\pm})$ is non-scattering and $u_{0}$ is real-valued, then by taking the real and imaginary parts of \eqref{eq:total-field-u-rho}, one sees that $u_{\rho\pm}$ must be real-valued. 
\end{rem*}

\begin{thm}\label{thm:nonscattering}
Let $k_{0} \ge 0$, $k_{\pm} \ge 0$, $\lambda_{\pm} > 0$ and $\mu_{\pm} \in \mathscr{E}'(\mR^{n})$. Suppose that there exists a solution $\tilde{u} \in \mE'(\mR^{n})$ of the two-phase problem \eqref{eq:main-2-phase} with 
\begin{equation}
\supp\,(\mu_{\pm}) \subset D_{\pm} := \{\pm \tilde{u}>0\}. \label{eq:support-condition-2-phase}
\end{equation}
If there exists an incident field $u_{0}$ of $(\Delta + k_{0}^{2})u_{0}=0$ in $\mR^{n}$ such that $u_{0}<0$ on $\partial D_{+} \cup \partial D_{-}$, 
then there exist contrasts $\rho_{\pm} \in L^{\infty}(D_{\pm})$ such that the pair of obstacles $(D_{\pm},\rho_{\pm})$ is non-scattering with respect to $u_{0}$. 
\end{thm}

\begin{rem*}
In the above theorem,  if the obstacles $D_{\pm}$ are ``touching'' each other,  i.e.,  $\partial D_{+}\cap \partial D_{-} \neq \emptyset$, then \emph{the common boundary is a two-phase free boundary}:   
\begin{equation}
\text{$u_{\rho\pm} > u_{0}$ in $D_{+}$ and $u_{\rho\pm} < u_{0}$ in $D_{-}$ near $\partial D_{+} \cap \partial D_{-}$,} \label{eq:two-phase-free-boundary}
\end{equation}
where $u_{\rho\pm}$ is given in \eqref{eq:total-field-u-rho}. In addition, $\lim_{D_{\pm}\ni x \rightarrow x_{0}} \rho_{\pm}(x)$ both exist with 
\begin{equation*}
\lim_{D_{+}\ni x \rightarrow x_{0}} \rho_{+}(x) = -\frac{\lambda_{+}}{u_{0}(x_{0})} > 0 ,\quad \lim_{D_{-} \ni x \rightarrow x_{0}} \rho_{-}(x) = -\frac{\lambda_{-}}{u_{0}(x_{0})} > 0 
\end{equation*} 
for each $x_{0}\in \partial D_{+} \cap \partial D_{-}$. 
\end{rem*}

\begin{rem*}[Existence of incident field $u_{0}$]
In general $\hbox{int}(\overline{ D_{+}\cup D_{-}})$ is not a Lipschitz domain. We still can construct such $u_{0}$ (can be even chosen to be Herglotz wave function \eqref{eq:Herglotz-wave-function} below) using \cite[Theorem~1.2]{KSS23PositiveHelmholtz} when $\partial D_{+}\cup\partial D_{-}$ is contained in a small set. 
\end{rem*}

\begin{proof}[Proof of Theorem~\ref{thm:nonscattering}]
From \eqref{eq:main-2-phase} and the support condition \eqref{eq:support-condition-2-phase}, one sees that there exists a neighborhood $U$ of $\partial (D_{+}\cup D_{-}) \equiv \partial D_{+} \cup \partial D_{-}$ in $\mR^{n}$ such that 
\begin{equation*}
(\Delta + k_{0}^{2}) \tilde{u} = h \chi_{D} \text{ in $U$} ,\quad \left. \tilde{u} \right|_{U\setminus D} = 0 ,\quad D = D_{+}\cup D_{-},
\end{equation*}
where 
\begin{equation}
h = - (k_{+}^{2}-k_{0}^{2})\tilde{u}_{+} + (k_{-}^{2}-k_{0}^{2})\tilde{u}_{-} + \lambda_{+}\chi_{D_{+}} - \lambda_{-}\chi_{D_{-}} \in L^{\infty}(D). \label{eq:constrast-h}
\end{equation}
By continuity of $\tilde{u}$ in $U$ and $\left.\tilde{u}\right|_{\partial D}=0$, one has $\abs{h} \ge \frac{1}{2} \min\{\lambda_{+},\lambda_{-}\} > 0$ near $\partial D$ in $D$. Now, the theorem (and the following remark) can be proved by following the exact same argument as in \cite[Theorem~2.4]{KSS23Minimization} (with $g\equiv 0$) and the discussions following the theorem.

\end{proof}

\subsubsection{\label{sec:Magnetic-field}
Connection with magnetic fields}
The Helmholtz equation is fundamental for understanding the spatial characteristics of electromagnetic fields, which provides a mathematical framework to describe how electromagnetic fields propagate and vary in space.

Here, we shall connect the concept developed in this paper to one of the waveguide mode, called the transverse-electric mode (TE-mode), which roughly means that there is no electric field in the direction of propagation, see \eqref{eq:conductivity} below. In this case, since there is only a magnetic field along the direction of propagation, sometimes we call this waveguide mode the $\bfH$-mode. One can refer e.g. the monograph \cite{KH15MaxwellBook} for mode details about this topic.

Let $\omega_{0} > 0$ denote a frequency, $\varepsilon_{0}$ represent the electric permittivity in a vacuum, and $\mu_{0}$ denote the magnetic permeability in a vacuum. The (time-harmonic) magnetic field $\mathbf{H} = (H_{1}, H_{2}, H_{3})$ in a medium with zero conductivity is governed by the equation
\begin{equation}
-\curl\,\left(\frac{1}{\calE(x)}\curl\,\bfH\right) + k_{0}^{2}\bfH = 0 \quad \text{in $\mR^{3}$,} \quad \text{with $\calE(x) = \frac{\varepsilon(x)}{\varepsilon_{0}} + \mathbf{i} \frac{\sigma(x)}{\omega\varepsilon_{0}}$,}\label{eq:Magnetic-field}
\end{equation}
where $k_{0} = \omega_{0}\sqrt{\varepsilon_{0}\mu_{0}} > 0$ is the wave number. In this equation, $\varepsilon(x)$ signifies the electric permittivity in the medium, and $\sigma(x)$ represents its conductivity; for further details, refer to \cite[(5.18)--(5.19)]{KG08Factorization} or \cite[(1.8)--(1.9)]{KH15MaxwellBook}.

We assume that $\calE = 1$ if and only if $\varepsilon=\varepsilon_{0}$ and $\sigma=0$ outside of a bounded domain. When we illuminate the inhomogeneity, supported on $\supp \,(\calE-1)$, using the incident magnetic field $\bfH_{0}$ that satisfies
\begin{equation*}
-\curl \,\curl \, \bfH_{0} + k_{0}^{2} \bfH_{0} = 0 \quad \text{in $\mathbb{R}^{3}$,}
\end{equation*}
then, under certain mild assumptions on $\varepsilon$ and $\sigma$ (refer to \cite[Theorem~5.5]{KG08Factorization}), there exists a unique scattered magnetic field $\bfH_{\rm sc}$ that  satisfies the equation
\begin{equation*}
-\curl \left( \frac{1}{\calE(x)} \curl\, \bfH_{\rm sc} \right) + k_{0}^{2} \bfH_{\rm sc} = - \curl \left( 1-\frac{1}{\calE(x)} \curl\, \bfH_{0} \right),
\end{equation*}
and the Silver-M\"{u}ller radiation condition 
\begin{equation*}
\curl \, \bfH_{\rm sc}(x) \times \hat{x} - \mathbf{i} k_{0} \bfH_{\rm sc}(x) = O(\abs{x}^{-2}) \quad \text{as $\abs{x}\rightarrow\infty$,}
\end{equation*}
uniformly on all direction $\hat{x}=x/|x| \in \calS^{2}$. 

\begin{rem*} 
By using the fact that ${\rm div}\,\curl \equiv 0$ and the curl-curl identity 
\begin{equation}
-\curl\,\curl\,=\Delta - \nabla \, {\rm div}, \label{eq:curl-curl-identity}
\end{equation}
it is easy to see that \eqref{eq:Magnetic-field} is equivalent to 
\begin{equation}
\Delta \bfH + \frac{\nabla \calE(x)}{\calE(x)} \times \curl\,\bfH + k_{0}^{2} \calE(x) \bfH=0 ,\quad {\rm div}\,\bfH = 0 \quad \text{in $\mR^{3}$.}
\label{eq:Magnetic-field-equivalent}
\end{equation}
It is  also  noteworthy  that the incident field satisfies the equation 
\begin{equation*}
\Delta \bfH_{0} + k_{0}^{2} \bfH_{0}=0 ,\quad {\rm div}\,\bfH_{0}=0 \quad \text{in $\mR^{3}$,}
\end{equation*}
and by direct computations, one  can easily see that $\Delta(x\cdot\bfH_{0}) + k_{0}^{2}(x\cdot\bfH_{0})=0$ in $\mR^{3}$. The curl-curl identity \eqref{eq:curl-curl-identity} can be extended for dimension $n \ge 2$ in terms of $n$-dimensional curl and its formal transpose. This even can be further extended to the symmetric tensors case in terms of Saint Venant operator \cite{IKS23UCPMRT}. 
\end{rem*}

In practical application, one usually illuminates the inhomogeneity using the superposition of plane waves, which called the Herglotz wave: 
\begin{equation*}
\bfH_{0}[p](\hat{x}) := \int_{\calS^{2}} p(\hat{z}) e^{\mathbf{i}k_{0}x\cdot\hat{z}} \, \rmd \hat{z} \quad \text{for all $p \in L_{\rm t}^{2}(\calS^{2})$ and for all $x\in\mR^{3}$,}
\end{equation*}
where 
\begin{equation*}
L_{\rm t}^{2}(\calS^{2}) := \left\{ \bfv \in (L^{2}(\calS^{2}))^{3} : \hat{x}\cdot \bfv(\hat{x}) = 0 , \hat{x}\in\calS^{2} \right\}.
\end{equation*}

The radiation condition for electromagnetic field is usually called the Silver-M\"{u}ller radiation condition, which is closely related to (see \cite[Corollary~2.53]{KH15MaxwellBook})  Sommerfeld radiation condition \eqref{sommerfeld}
and the far-field operator is analogously defined by the far-field amplitude of the scattered field.
In fact, one can reconstruct $\supp\,(\calE-1)$ from the far-field amplitude \cite{KG08Factorization}.

In the case when both $\calE$ and $\bfH$ are cylindrical, i.e. independent to the variable $x_{3}$, we see that the third component $H_{3}(x')$ of $\bfH$ in \eqref{eq:Magnetic-field} satisfies the isotropic elliptic equation 
\begin{equation}
{\rm div}'\, \left( \frac{1}{\calE(x')} \nabla' H_{3} \right) + k_{0}^{2} H_{3} = 0 \quad \text{in $\mR^{2}$,} \label{eq:conductivity}
\end{equation}
where $\nabla'$ and ${\rm div}'$ are gradient and divergence operator on $\mR^{2}$. In this case, we usually not interested in the first two components $H_{1}$ and $H_{2}$, and we simply put $H_{1}\equiv H_{2} \equiv 0$, and this situation is called the \emph{magnetic mode} ($\bfH$-mode) or \emph{transverse-electric mode} (TE-mode) \cite[page~11]{KH15MaxwellBook}.

If $\calE \in C^{2}(\mR^{2})$ and is real-valued (iff $\sigma \equiv 0$), one can rewrite \eqref{eq:conductivity} as the Helmholtz equation: 
\begin{equation} 
(\Delta' + k_{0}^{2} + q(x'))u=0 \quad \text{in $\mR^{2}$} \label{eq:Helmholtz-magnetic}
\end{equation} 
with $u = (\calE(x'))^{\frac{1}{2}} H_{3}(x')$ and $q(x')= - (\calE(x'))^{\frac{1}{2}} \Delta' \left( (\calE(x'))^{-\frac{1}{2}} \right)$, where $\Delta' = {\rm div}' \nabla'$ is the Laplacian on $\mR^{2}$, see e.g. \cite[(0.2)--(0.3)]{Nachman19962DCalderon}. 
We can formulate similar inverse problems involving the reconstruction of $q$ from the knowledge of the far-field operator \eqref{eq:far-field-operator}. After recovering $q$, we then finally recover $\calE$ by solving the following elliptic boundary-value problem: 
\begin{equation} 
\Delta' \left( (\calE(x'))^{-\frac{1}{2}} \right) + q(x') (\calE(x'))^{-\frac{1}{2}} = 0 \text{ in $B_{R}$} ,\quad \left. (\calE(x'))^{-\frac{1}{2}} \right|_{\partial B_{R}}=1 \label{eq:reconstruct-calE}
\end{equation}
by choosing suitable large $R>0$. We can construct $q = \rho_{+}\chi_{D_{+}} - \rho_{-}\chi_{D_{-}}$ as described in Theorem~\ref{thm:nonscattering}, and then construct $\calE = \calE(x_{1},x_{2})$ by solving \eqref{eq:reconstruct-calE}. Formally, this is non-scattering with respect to some incident $\bfH$-mode/TE-mode magnetic field.

\subsubsection{Some related application}

We now revisit the Helmholtz equation, selecting $u^{\rm inc}$ as the superposition of the plane incident wave, expressed as the Herglotz wave function:
\begin{equation}
u^{\rm inc}[f] = \int_{\mathcal{S}^{n-1}} e^{\mathbf{i}kx\cdot\theta} f(\theta) \, \rmd \theta \label{eq:Herglotz-wave-function}
\end{equation}
where $f \in L^{2}(\mathcal{S}^{n-1})$. Consequently, we consider the far-field operator:
\begin{equation}
f\in L^{2}(\mathcal{S}^{n-1}) \mapsto \int_{\mathcal{S}^{n-1}}u^{\infty}(\theta',\theta)f(\theta) \, \rmd \theta. \label{eq:far-field-operator}
\end{equation}
Here, $u^{\infty}(\theta',\theta)$ represents the far-field of the scattered field corresponding to the incident plane wave $e^{\mathbf{i}kx\cdot\theta}$. 

Combining results from \cite{Nachman19962DCalderon, sylvester1987global}, if $k^{2}$ is not a Dirichlet eigenvalue of $-\Delta$ on $D$, it can be shown that $\rho\chi_{D}$ can be uniquely determined from the far-field operator \eqref{eq:far-field-operator}. Refer to \cite{hahner2001new} for a log-type stability estimate, which is nearly optimal \cite{isaev2013exponential}. See also \cite[Appendix~B]{FKW23BayesDeterministic}.
In practice, obtaining only finitely many measurements $\left\{u^{\rm inc}[f_{i}]:i=1,\cdots,N\right\}$ is feasible. 
However, based on nonscattering results in Theorem~\ref{thm:nonscattering} above (also see \cite{KLSS22QuadratureDomain, KSS23Minimization}), it is generally impossible to determine $\rho\chi_{D}$ solely from a single measurement $u^{\rm inc}[f_{1}]$.
Thus, one should not expect to \emph{always} determine $\rho\chi_{D}$ from finitely many measurements $\left\{u^{\rm inc}[f_{i}]:i=1,\cdots,N\right\}$.

Intuitively, one can approximate the far-field operator \eqref{eq:far-field-operator} using $\left\{u^{\rm inc}[f_{i}]:i=1,\cdots,N\right\}$ for large $N$. For instance, choosing $f_{i}$ as the eigenfunction of the Laplace-Beltrami operator on $\mathcal{S}^{n-1}$ is a possible approach. This intuition can be validated in a probabilistic sense (with randomly chosen samples $f_{1},\cdots,f_{N}$ with a large $N$), as seen in \cite{FKW23BayesDeterministic}. In simple terms, while one might fail to determine $\rho\chi_{D}$ from the knowledge of $\left\{u^{\rm inc}[f_{i}]:i=1,\cdots,N\right\}$ for randomly chosen $f_{1},\cdots,f_{N}$, the probability of such a situation occurring decreases as the sample size $N$ increases.

\section{Multiphase problem through  minimization}

\addtocontents{toc}{\SkipTocEntry}
\subsection{Main results}

We now delve into the exploration of the existence of two-phase $(k_{1},k_{2})$-quadrature domains \eqref{eq:main-2-phase}. Let $\Omega \subset \mathbb{R}^{n}$ be a bounded Lipschitz domain in $\mathbb{R}^{n}$ ($n \ge 2$). The function space $C_{c}^{\infty}(\Omega)$ consists of $C^{\infty}(\mathbb{R}^{n})$ functions supported in $\Omega$, and $H_{0}^{1}(\Omega)$ is the completion of $C_{c}^{\infty}(\Omega)$ with respect to the $H^{1}(\Omega)$-norm.

For each $f_{1},\ldots,f_{m} \in L^{\infty}(\mathbb{R}^{n})$ and $\bfk=(k_{1},\ldots,k_{m})$ with $k_{1},\ldots,k_{m}\ge 0$, we consider the functional
\begin{equation}
\mJ_{\bfk}(\bfu) := \sum_{i=1}^{m} \mJ_{k_{i}}(u_{i}) ,\quad \mJ_{k_{i}}(u_{i}) := \int_{\Omega} \left( \abs{\nabla u_{i}(x)}^{2} - k_{i}^{2}\abs{u_{i}(x)}^{2} - 2f_{i}(x)u_{i}(x) \right) , \rmd x \label{eq:functional-J}
\end{equation}
for $\bfu=(u_{1},\ldots,u_{m})\in (H^{1}(\Omega))^{m}$. 
We define
\begin{equation*}
\begin{aligned}
\mK_{m}(\Omega) &:= \left\{ (u_{1},\cdots,u_{m}) \in (H_{0}^{1}(\Omega))^{m} : \text{$u_{i}\ge 0$ for all $i=1,\cdots,m$} \right\}, \\
\mS_{m}(\Omega) &:= \left\{ (u_{1},\cdots,u_{m}) \in \mK_{m}(\Omega) : \text{$u_{i}\cdot u_{j}= 0$ for all $i\neq j$} \right\}. 
\end{aligned}
\end{equation*}
Similar to \cite{CTV05SegregationProblem}, we refer to the elements in $\mS_{m}(\Omega)$ as \emph{segregated states}. A state $\bfu=(u_{1},\ldots,u_{m})\in (H^{1}(\Omega))^{m}$ is called \emph{segregated} if $u_{i}\cdot u_{j}=0$ for all $i\neq j$.

The primary focus of this paper is to investigate the following minimization problem:
\begin{equation}
\text{Minimize $\mJ_{\bfk}(\bfu)$ subject to segregated states $\bfu\in \mS_{m}(\Omega)$.} \label{eq:minimizing-m-phase-problem}
\end{equation}

The situation where $\bfk \equiv \mathbf{0}$ was examined in \cite{AS16MultiPhaseQD}, and an application from control theory was presented. Also  in \cite{KSS23Minimization}, the functional $\mJ_{k_{i}}$ was investigated for sufficiently small $k_{i} > 0$.

By using \cite[Lemma~3.1]{KSS23Minimization}, it is easy to see that $\mJ_{\bfk}$ is unbounded from below in $\mS_{m}(\Omega)$ if $k_{i}>k_{*}$ for some $i \in \{1,\cdots,m\}$, where
\begin{equation}
k_{*}^{2}(\Omega) := \inf_{\phi\in C_{c}^{\infty}(\Omega),\phi\not\equiv 0} \frac{\norm{\nabla\phi}_{L^{2}(\Omega)}^{2}}{\norm{\phi}_{L^{2}(\Omega)}^{2}} > 0 \label{eq:principal-eigenvalue}
\end{equation}
is the first Dirichlet eigenvalue of $\Omega$. When $0\le k_{i}<k_{*}$ for all $i=1,\cdots,m$, by using \cite[Lemma~3.4]{KSS23Minimization}, we know that $\mJ_{\bfk}$ is weakly lower semi-continuous on $(H_{0}^{1}(\Omega))^{m}$. Since the set $\mS_{m}(\Omega)$ is closed in $(H_{0}^{1}(\Omega))^{m}$, by following the standard arguments of calculus of variations (as in \cite[Proposition~3.6]{KSS23Minimization}) one can show that
\begin{equation}
\text{there exists a minimizer $\bfu_{*}$ of the functional $\mJ_{\bfk}$ in $\mS_{m}(\Omega)$.} \label{eq:existence-global-minimizer}
\end{equation} 
We show that the difference $u_{*,i}-u_{*,j}$ locally satisfies the two-phase obstacle equation. 

\begin{thm}\label{thm:local-properties}
Let $\Omega$ be a bounded Lipschitz domain in $\mR^{n}$, let $0 \le k_{i} < k_{*}$ and $f_{i}\in L^{\infty}(\Omega)$ for all $i=1,\cdots,m$. If $\bfu_{*}=(u_{*,1},\cdots,u_{*,m})$ is a \emph{segregated ground state} of the energy functional $\mJ_{\bfk}$, i.e. a minimizer of the functional $\mJ_{\bfk}$ in $\mS_{m}(\Omega)$, then
\begin{equation}
\Delta(u_{*,i}-u_{*,j}) + k_{i}^{2}u_{*,i} - k_{j}^{2}u_{*,j} = -f_{i}\chi_{\{u_{*,i}>0\}} + f_{j}\chi_{\{u_{*,j}>0\}} \quad \text{in $\Omega \setminus \bigcup_{k\neq i,j} \overline{\Omega_{k}}$}, \label{eq:local-2phase}
\end{equation}
where $\Omega_{i} = \{u_{*,i}>0\}$ for all $i=1,\cdots,m$.
\end{thm}

\begin{rem*}
As mentioned above, we assume $\Omega$ has Lipschitz boundary in order to guarantee \eqref{eq:existence-global-minimizer}, see also \cite[Remark~3.5]{KSS23Minimization}. 
\end{rem*}

When $m=2$, the functional $\mJ_{k_{1},k_{2}} \equiv \mJ_{\bfk}$ reads 
\begin{equation*}
\mJ_{k_{1},k_{2}}(u_{1},u_{2}) = \mJ_{k_{1}}(u_{1}) + \mJ_{k_{2}}(u_{2}), 
\end{equation*}
and the minimization problem \eqref{eq:minimizing-m-phase-problem} reads: 
\begin{equation}
\text{minimize $\mJ_{k_{1},k_{2}}(u_{1},u_{2})$ subject to $(u_{1},u_{2})\in \mS_{2}(\Omega)$.} \label{eq:minimizing-2phase-problem}
\end{equation}
Similar to \cite{EPS11TwoPhaseQD}, we consider the functional
\begin{equation*}
\tilde{\mJ}_{k_{1},k_{2}}(U) = \int_{\Omega} \left( \abs{\nabla U}^{2} - k_{1}^{2}\abs{U_{+}}^{2} - k_{2}^{2}\abs{U_{-}}^{2} - 2f_{1}U_{+} - 2f_{2}U_{-} \right) \, \rmd x
\end{equation*}
with $f_{1},f_{2} \in L^{\infty}(\Omega)$, and the minimization problem
\begin{equation}
\text{minimize $\tilde{\mJ}_{k_{1},k_{2}}(U)$ subject to $U \in H_{0}^{1}(\Omega)$.} \label{eq:minimizing-2phase-problem-equivalent}
\end{equation}

In fact, the minimizing problems \eqref{eq:minimizing-2phase-problem} and \eqref{eq:minimizing-2phase-problem-equivalent} are equivalent in the following sense:
\begin{equation}
\parbox{0.8\linewidth}{If $\tilde{u}$ is a minimizer of the functional $\tilde{\mJ}_{k_{1},k_{2}}$ in $H_{0}^{1}(\Omega)$, then $(\tilde{u}_{+},\tilde{u}_{-})$ is a minimizer of the functional $\mJ_{k_{1},k_{2}}$ in $\mS_{2}(\Omega)$. Conversely, if $(u_{*,1},u_{*,2})$ is a minimizer of the functional  $\mJ_{k_{1},k_{2}}$ in $\mS_{2}(\Omega)$, then $\tilde{u}_{*} := u_{*,1}-u_{*,2}$ is a minimizer of the functional $\tilde{\mJ}_{k_{1},k_{2}}$ in $H_{0}^{1}(\Omega)$. } \label{eq:equivalent-2phase}
\end{equation}
This can be proved by following the arguments in \cite[Theorem~2]{AS16MultiPhaseQD} and the observation 
\begin{equation*}
w_{1}=(w_{1}-w_{2})_{+} ,\quad w_{2}=(w_{1}-w_{2})_{-} \quad \text{for all $(w_{1},w_{2})\in \mS_{2}(\Omega)$. }
\end{equation*}
When $0 \le k_{1},k_{2} < k_{*}$, by using \eqref{eq:existence-global-minimizer}, one immediately sees that there exists a minimizer $\tilde{u}$ of the functional $\tilde{\mJ}_{k_{1},k_{2}}$ in $H_{0}^{1}(\Omega)$. Consequently, from Theorem~\ref{thm:local-properties}, we can easily conclude that: If $\tilde{u}$ is a minimizer of the functional $\tilde{J}_{k_{1},k_{2}}$ in $H_{0}^{1}(\Omega)$, then 
\begin{equation}
\Delta \tilde{u} + k_{1}^{2} \tilde{u}_{+} - k_{2}^{2} \tilde{u}_{-} = -f_{1}\chi_{\{\tilde{u}>0\}} + f_{2}\chi_{\{\tilde{u}<0\}} \quad \text{in $\Omega$.} \label{eq:global-minimizer-2phase}
\end{equation} 
Therefore the two-phase problem \eqref{eq:main-2-phase} is a special case of the multi-phase problem \eqref{eq:local-2phase}. We also exhibit some interesting points in Appendix~\ref{appen:saddle-point}.

By using approximation theorem, we are also able to extend the existence result for the solution of the local two-phase problem \eqref{eq:local-2phase} for $f_{i} = \mu_{i} - \lambda_{i}$ when $\mu_{i}$ is a measure.

\begin{thm}\label{thm:local-properties-unbounded-f}
Let $\Omega$ be a bounded Lipschitz domain in $\mR^{n}$, let $0 \le k_{i} < k_{*}$, let $\lambda_{i}$ be positive $L^{\infty}(\Omega)$ functions with $\lambda_{i} \ge c > 0$ in $\Omega$, and let $\mu_{i} \in \mE'(\Omega)$. Then there exists at least one solution $\bfu_{*}=(u_{*,1},\cdots,u_{*,m})$ of \eqref{eq:local-2phase} with $f_{i}=\mu_{i}-\lambda_{i}$. 
\end{thm}

In \cite[Proposition~3.6]{KSS23Minimization} it is shown 
 that there exists a minimizer $v_{*,i}$ of the functional $\mJ_{k_{i}}$ in $\mK_{1}(\Omega)$. In this case, by using the Euler-Lagrange equation, one can prove that such a minimizer $v_{*,i}$ of the functional $\mJ_{k_{i}}$ is unique in $\mK_{1}(\Omega)$. 
This implies that\footnote{Given $\bfv_{*}'=(v_{*,1}',\cdots,v_{*,m}') \in \mK_{m}(\Omega)$ with $\bfv_{*}'\neq\bfv_{*}$, i.e. $v_{*,j}'\neq v_{*,j}$ for some $j$, one sees that $v_{*,j}'$ must not the minimizer of $\mJ_{k_{j}}$, therefore $\mJ_{k_{j}}(v_{*,j}') > \mJ_{k_{j}}(v_{*,j})$, which implies $\mJ_{\bfk}(\bfv_{*}') > \mJ_{\bfk}(\bfv_{*})$, i.e. the minimizer of $\mJ_{\bfk}$ in $\mK_{m}(\Omega)$ is unique.}
\begin{equation}
\text{$\bfv_{*} = (v_{*,1},\cdots,v_{*,m})$ is the unique minimizer of the functional $\mJ_{\bfk}$ in $\mK_{m}(\Omega)$.} \label{eq:minimizer-non-orthogonal}
\end{equation}
However, at this point, we do not know whether $\mathbf{v}{*}$ is segregated (i.e., $v_{i}\cdot v_{j}=0$ for all $i\neq j$) or not. We can compare the supports of minimizers in \eqref{eq:minimizing-m-phase-problem} with $\text{supp}(v_{*,i})$ as presented in the following theorem.

\begin{thm}\label{thm:support-minimizer}
Let $\Omega$ be a bounded Lipschitz domain in $\mR^{n}$, let $k_{1}=\cdots=k_{m}=k$ with $0 \le k < k_{*}$ and let $f_{i} \in L^{\infty}(\Omega)$ for all $i=1,\cdots,m$. Then for each minimizer $\bfu_{*}=(u_{*,1},\cdots,u_{*,m})$ of the functional $\mJ_{\bfk}$ in $\mS_{m}(\Omega)$ one has  
\begin{equation}
\supp\,(u_{*},i) \subset \supp\,(v_{*},i) \quad \text{for all $i=1,\cdots,m$,}
\end{equation}
where $\bfv_{*}=(v_{*,1},\cdots,v_{*,m})$ is given in \eqref{eq:minimizer-non-orthogonal}. 
\end{thm}

Suppose that for each $i=1,2,\cdots,m$ we are given the non-negative distribution $\mu_{i}$ which is sufficiently concentrated near $x_{i}\in\mR^{n}$ in the sense of 
\begin{equation*}
\text{$\mu=0$ outside $B_{\epsilon_{i}}(x_{i})$}
\end{equation*}
and 
\begin{equation*}
\mu_{i}(\mR^{n}) > C_{n} \lambda_{i} \epsilon_{i} ,\qquad C_{n} \ge 2^{n} \frac{(3\pi)^{\frac{n}{2}}}{\Gamma(1+\frac{n}{2})} \frac{J_{\frac{n}{2}}(j_{\frac{n-2}{2},1})}{J_{\frac{n}{2}}(j_{\frac{n-2}{2},1}/3)}
\end{equation*}
for some constants $\epsilon_{i}>0$ and $\lambda_{i}>0$, 
where $\Gamma$ is the standard Gamma function, $J_{\alpha}$ is the Bessel functions of order $\alpha$ of the first kind, and $j_{\alpha,m}$ is the $m^{\rm th}$ positive root of $J_{\alpha}$. We now fix a parameter $0 < \beta < j_{\frac{n-2}{2},1}$. By using \cite[Theorem~7.6]{KSS23Minimization}, for each $k>0$ satisfying 
\begin{equation*}
k < \min \left\{ \frac{1}{3} , \left( C_{n,\beta}\frac{\lambda_{i}}{\mu(\mR^{n})} \right)^{\frac{1}{n}} \right\} ,\qquad C_{n,\beta} = \left( \frac{4\pi}{3} \right)^{\frac{n}{2}} \beta^{\frac{n}{2}} J_{\frac{n}{2}}(\beta) \frac{J_{\frac{n}{2}}(\frac{2}{3}j_{\frac{n-2}{2},1})}{J_{\frac{n}{2}}(j_{\frac{n-1}{2},1})},
\end{equation*}
there exists a non-negative function $v_{*,i} \in C_{\rm loc}^{0,1}(\mR^{n})$ such that 
\begin{equation*}
(\Delta + k^{2})v_{*,i} = - \tilde{\mu}_{i} + \lambda_{i}\chi_{Q_{i}} ,\quad Q_{i} = \{ v_{*,i}>0 \} ,\quad \tilde{\mu}_{i}=\mu_{i}*\phi_{2\epsilon_{i}} 
\end{equation*}
with the support condition 
\begin{equation}
\supp\,(\tilde{\mu}_{i}) \subset Q_{i} ,\quad \overline{Q_{i}} \subset B_{\beta k^{-1}}(x_{i}) , \label{eq:support-condition}
\end{equation}
where
\begin{equation*}
\phi_{2\epsilon} = (c_{n,k,2\epsilon}^{\rm MVT})^{-1}\chi_{B_{2\epsilon}} ,\quad c_{n,k,2\epsilon}^{\rm MVT} = (2\pi)^{\frac{n}{2}} k^{-\frac{n}{2}} (2\epsilon)^{\frac{n}{2}} J_{\frac{n}{2}}(2k\epsilon).
\end{equation*}

Here, $Q_{i}$ is a (1-phase) $k$-quadrature domain corresponding to $\mu_{i}$ and positive constant $\lambda_{i}$. In addition, such $v_{*,i}$ is also the unique minimizer of $\mJ_{k}$ with $f_{i}=\mu_{i}-\lambda_{i}\chi_{Q_{i}}$ in $\mK_{1}(\Omega)$. Hence we see that $\bfv_{*}=(v_{*,1},\cdots,v_{*,m})$ is the unique minimizer of the functional $\mJ_{\bfk}$ with $\bfk=(k,k,\cdots,k)$ and $f_{i}=\mu_{i}-\lambda_{i}\chi_{Q_{i}}$ in $\mK_{m}(\Omega)$. When $\mu$ is bounded, the above also holds true by replacing $\tilde{\mu}$ with  $\mu$. 

It is important to notice that $Q_{i}$ may not be disjoint even in the case when $\supp\,(\mu_{i})\cap\supp\,(\mu_{j})=\emptyset$ for all $i\neq j$. In this case, we need to shrink $Q_{i}$ into $\{u_{*,i}>0\}$ in the sense of Theorem~\ref{thm:support-minimizer}.

An essential component of the theory of quadrature domains involves the relationship:
\begin{equation}
\text{supp}(\mu_{i}) \subset \{u_{*,i} > 0\}. \label{eq:support-condition-multiphase}
\end{equation}
From Theorem~\ref{thm:support-minimizer}, it becomes evident that a prerequisite for \eqref{eq:support-condition-multiphase} is expressed by:
\begin{equation*}
\supp\,(\mu_{i}) \subset Q_{i}:=\{v_{*,i}>0\}.
\end{equation*}
In this scenario, for each $i=1,\ldots,n$, it can be observed that $Q_{i}$ represents a one-phase $k$-quadrature domain corresponding to $\mu_{i}$. This observation prompts the formulation of sufficient conditions for 
\eqref{eq:support-condition-multiphase} in terms of 1-phase $k$-quadrature domains.

The following theorem exhibits some sufficient condition to guarantee the following weaker support condition: 
\begin{equation}
\supp\,(\mu_{i}) \subset \supp\,(u_{*,i}), \label{eq:support-condition-0}
\end{equation}
where $\bfu_{*}=(u_{*,1},\cdots,u_{*,m})$ is a minimizer of the functional $\mJ_{\bfk}$ in $\mS_{m}(\Omega)$: 

\begin{thm}\label{thm:m-phase-quadrature-domains}
Let $\Omega$ be a bounded Lipschitz domain and let $0<k<k_{*}$. For each $i=1,\cdots,m$, let $\mu_{i} \in L^{\infty}(\mR^{n})$ with $\supp\,(\mu_{i}) \subset \Omega$ and let $Q_{i}=\left\{v_{*,i}>0\right\}$, where $v_{*,i}$ is the unique minimizer of $\mJ_{k_{i}}$ in $\mK_{1}(\Omega)$ with $k_{i}=k$. Suppose that 
\begin{equation}
\overline{Q_{i}} \cap \supp\,(\mu_{j}) = \emptyset \quad \text{for all $i\neq j$,} \label{eq:disjoint-condition}
\end{equation}
and let $\bfk=(k,k,\cdots,k)$ and $f_{i} = \mu_{i}-\lambda_{i}$ for $i=1,\cdots,m$. If 
\begin{equation}
\inf_{x \in U_{i}}\mu_{i}(x) > \lambda_{i} \quad \text{holds for all $i=1,\cdots,m$,} \label{eq:non-degenerate-condition}
\end{equation}
for some open sets $U_{i} \subset \supp\,(\mu_{i})$, then all minimizers of $\mJ_{\bfk}$ in $\mS_{m}(\Omega)$ satisfy the support condition  
\begin{equation}
\overline{U_{i}} \subset \supp\,(u_{*,i}) \quad \text{for all $i=1,\cdots,m$}. \label{eq:support-condition-1}
\end{equation}
\end{thm}

\begin{rem*}
If $\supp\,(\mu_{i})=\overline{{\rm int}\,(\supp\,(\mu_{i}))}$, then we can guarantee \eqref{eq:support-condition-0} by choosing $U_{i}={\rm int}\,(\supp\,(\mu_{i}))$. 
\end{rem*}

\addtocontents{toc}{\SkipTocEntry}
\subsection{Proofs of the theorems}

By modifying the ideas in \cite[Proposition~1]{AS16MultiPhaseQD}, we now prove Theorem~\ref{thm:local-properties}. 

\begin{proof}[Proof of Theorem~\ref{thm:local-properties}]
Let $\psi \in C_{c}^{\infty}$ be non-negative such that $\supp\,(\psi) \subset \Omega\setminus \bigcup_{k\neq i,j} \overline{\Omega_{k}}$. Given any $\epsilon>0$, we define 
\begin{equation*}
Q_{\epsilon} := \left\{ x\in\Omega : u_{*,i}-u_{*,j} \le \epsilon \psi \right\}. 
\end{equation*}
If we define $\bfz := (z_{1},z_{2},\cdots,z_{m})$ as 
\begin{equation*}
\text{$z_{\ell} := u_{*,\ell}$ if $\ell\neq i,j$}, \quad z_{i}:= (u_{*,i}-u_{*,j}-\epsilon\psi)_{+} \quad \text{and} \quad z_{j} := (u_{*,i}-u_{*,j}-\epsilon\psi)_{-},
\end{equation*}
then $\bfz \in \mS_{m}(\Omega)$. Since $\left. u_{*,j} \right|_{\Omega\setminus Q_{\epsilon}}=0$, $u_{*,i} = (u_{*,i}-u_{*,j})_{+}$ and $u_{*,i} = (u_{*,i}-u_{*,j})_{-}$, we obtain 
\begin{equation}
0 \le \mJ_{\bfk}(\bfz) - \mJ_{\bfk}(\bfu_{*}) = I_{1} + I_{2}, \label{eq:local-decomposition}
\end{equation}
where 
\begin{equation*}
\begin{aligned}
I_{1} & = \int_{\Omega}  \left( \abs{\nabla(u_{*,i}-u_{*,j}-\epsilon\psi)_{+}}^{2} - \abs{\nabla u_{*,i}}^{2} \right) + \left( \abs{\nabla(u_{*,i}-u_{*,j}-\epsilon\psi)_{-}}^{2} - \abs{\nabla u_{*,j}}^{2} \right) \, \rmd x \\
& \quad + 2 \int_{\Omega} f_{i}(u_{*,i}-(u_{*,i}-u_{*,j}-\epsilon\psi)_{+}) + f_{j}(u_{*,j}-(u_{*,i}-u_{*,j}-\epsilon\psi)_{-}) \, \rmd x 
\end{aligned}
\end{equation*}
and 
\begin{equation*}
I_{2} = - \int_{\Omega} k_{i}^{2} \left( \abs{(u_{*,i}-u_{*,j}-\epsilon\psi)_{+}}^{2} - \abs{u_{*,i}}^{2} \right) + k_{j}^{2} \left( \abs{(u_{*,i}-u_{*,j}-\epsilon\psi)_{-}}^{2} - \abs{u_{*,j}}^{2} \right) \, \rmd x
\end{equation*}
Following the exactly same arguments as in \cite[Proposition~1]{AS16MultiPhaseQD}, one can show that 
\begin{equation}
\begin{aligned}
I_{1} & \le -2\epsilon \int_{\Omega\setminus \bigcup_{k\neq i,j} \overline{\Omega_{k}}} \nabla (u_{*,i}-u_{*,j}) \cdot \nabla \psi + \epsilon^{2} \int_{\Omega\setminus \bigcup_{k\neq i,j} \overline{\Omega_{k}}} \abs{\nabla\psi}^{2} \, \rmd x \\
& \quad + 2\epsilon \int_{\Omega\setminus \bigcup_{k\neq i,j} \overline{\Omega_{k}}} f_{i} \chi_{\{u_{*,i}>u_{*,j}\}} \psi \, \rmd x - 2\epsilon \int_{\Omega\setminus \bigcup_{k\neq i,j} \overline{\Omega_{k}}} f_{j} \chi_{\{u_{*,i}<u_{*,j}\}} \psi \, \rmd x + o(\epsilon)
\end{aligned} \label{eq:local-I1}
\end{equation}
On the other hand, we see that 
\begin{equation*}
\begin{aligned}
I_{2} & = - \int_{\Omega\setminus Q_{\epsilon}} k_{i}^{2} \left( \abs{u_{*,i}-(u_{*,j}+\epsilon\psi)}^{2} - \abs{u_{*,i}}^{2} \right) \, \rmd x \\
&\quad + \int_{Q_{\epsilon}} k_{i}^{2}\abs{u_{*,i}}^{2} \, \rmd x - \int_{Q_{\epsilon}} k_{j}^{2} \left( \abs{u_{*,j}-(u_{*,i}-\epsilon\psi)}^{2} - \abs{u_{*,j}}^{2} \right) \, \rmd x \\
& = 2\epsilon \int_{\Omega\setminus Q_{\epsilon}} k_{i}^{2} u_{*.i}\psi \, \rmd x - 2\epsilon \int_{Q_{\epsilon}} k_{j}^{2}u_{*,j}\psi \, \rmd x \\
& \quad + \int_{Q_{\epsilon}} k_{i}^{2}\abs{u_{*,i}}^{2} \, \rmd x - \int_{Q_{\epsilon}} k_{j}^{2}\abs{u_{*,i}-\epsilon\psi}^{2} \, \rmd x - \int_{\Omega\setminus Q_{\epsilon}} k_{i}^{2} \abs{u_{*,j}+\epsilon\psi}^{2} \, \rmd x \\
& = 2\epsilon \int_{\Omega\setminus \bigcup_{k\neq i,j} \overline{\Omega_{k}}} (k_{i}^{2}\chi_{\Omega\setminus Q_{\epsilon}} + k_{j}^{2}\chi_{Q_{\epsilon}}) u_{*.i}\psi \, \rmd x - 2\epsilon \int_{\Omega\setminus \bigcup_{k\neq i,j} \overline{\Omega_{k}}} (k_{j}^{2}\chi_{Q_{\epsilon}} + k_{i}^{2}\chi_{\Omega\setminus Q_{\epsilon}})u_{*,j}\psi \, \rmd x \\
& \quad + (k_{i}^{2}-k_{j}^{2}) \int_{Q_{\epsilon}} \abs{u_{*,i}}^{2} \, \rmd x + \epsilon^{2}(k_{i}^{2}+k_{j}^{2}) \int_{Q_{\epsilon}} \abs{\psi}^{2} \, \rmd x 
\end{aligned}
\end{equation*}
Since $u_{*,i}\cdot u_{*,j}=0$, then 
\begin{equation*}
\int_{Q_{\epsilon}} \abs{u_{*,i}}^{2} \, \rmd x = \int_{x\in\Omega, u_{*,i}(x)\le\epsilon\psi(x)} \abs{u_{*,i}}^{2} \, \rmd x \le \epsilon^{2} \int_{\Omega\setminus \bigcup_{k\neq i,j} \overline{\Omega_{k}}} \abs{\psi}^{2} \, \rmd x,
\end{equation*}
as well as 
\begin{equation*}
\begin{aligned}
& (k_{i}^{2}\chi_{\Omega\setminus Q_{\epsilon}} + k_{j}^{2}\chi_{Q_{\epsilon}}) u_{*,i} = (k_{i}^{2}\chi_{\{u_{*,i}>\epsilon\psi\}} + k_{j}^{2}\chi_{\{u_{*,i}\le\epsilon\psi\}}) u_{*,i} \\
& (k_{j}^{2}\chi_{Q_{\epsilon}} + k_{i}^{2}\chi_{\Omega\setminus Q_{\epsilon}}) u_{*,j} = k_{j}^{2} u_{*,j}.
\end{aligned}
\end{equation*}
From this, we reach  
\begin{equation}
\begin{aligned}
I_{2} &\le 2\epsilon \int_{\Omega\setminus \bigcup_{k\neq i,j} \overline{\Omega_{k}}} (k_{i}^{2}\chi_{\{u_{*,i}>\epsilon\psi\}} + k_{j}^{2}\chi_{\{u_{*,i}\le\epsilon\psi\}}) u_{*,i}\psi \, \rmd x \\
& \quad - 2\epsilon \int_{\Omega\setminus \bigcup_{k\neq i,j} \overline{\Omega_{k}}} k_{j}^{2} u_{*,j}\psi \, \rmd x + C \epsilon^{2} \int_{\Omega\setminus \bigcup_{k\neq i,j} \overline{\Omega_{k}}} \abs{\psi}^{2} \, \rmd x.
\end{aligned} \label{eq:local-I2}
\end{equation}
Combining \eqref{eq:local-decomposition}, \eqref{eq:local-I1} and \eqref{eq:local-I2}, we divide the resulting inequality by $2\epsilon$ and then taking the limit $\epsilon\rightarrow 0$, we reach 
\begin{equation*}
\Delta(u_{*,i}-u_{*,j}) + k_{i}^{2}u_{*,i} - k_{j}^{2}u_{*,j} \le -f_{i}\chi_{\{u_{*,i}>0\}} + f_{j}\chi_{\{u_{*,j}>0\}} \quad \text{in $\Omega \setminus \bigcup_{k\neq i,j} \overline{\Omega_{k}}$}.
\end{equation*}
Finally, by interchanging the role of $i$ and $j$ we conclude \eqref{eq:local-2phase}. 
\end{proof}

We are now in position to prove Theorem~\ref{thm:local-properties-unbounded-f}, which can be done similarly to \cite[Theorem~5]{AS16MultiPhaseQD}. 

\begin{proof}[Proof of Theorem~\ref{thm:local-properties-unbounded-f}]
For each $\mu_{i}$, we choose the sequence $\mu_{i}^{n} \in C_{c}^{\infty}(\Omega)$ such that $\mu_{i}^{n} \rightarrow \mu_{i}$ in $\mE'(\Omega)$. We choose $f_{i}^{n} = \mu_{i}^{n} - \lambda_{i}$, and it is clear that $f_{i}^{n}\rightarrow f_{i} := \mu_{i} - \lambda_{i}$ in $\mE'(\Omega)$. We consider the functional 
\begin{equation*}
\mJ_{\bfk}^{n}(\bfu) = \sum_{i=1}^{m} \mJ_{k_{i}}^{n}(u_{i}) ,\quad \mJ_{k_{i}}^{n}(u_{i}) = \int_{\Omega} \left( \abs{\nabla u_{i}}^{2} - k^{2}\abs{u_{i}}^{2} - 2f_{i}^{n}u_{i} \right) \, \rmd x. 
\end{equation*}
By \eqref{eq:existence-global-minimizer}, for each $n\in\mN$ there exists a minimizer $\bfu_{*}^{n}$ of the functional $\mJ_{\bfk}^{n}$ in $\mS_{m}(\Omega)$. By using Theorem~\ref{thm:local-properties}, such minimizer satisfies 
\begin{equation*} 
\Delta(u_{*,i}-u_{*,j}) + k_{i}^{2}u_{*,i} - k_{j}^{2}u_{*,j} = -f_{i}^{n}\chi_{\{u_{*,i}>0\}} + f_{j}^{n}\chi_{\{u_{*,j}>0\}} \quad \text{in $\Omega \setminus \bigcup_{k\neq i,j} \overline{\Omega_{k}}$}. 
\end{equation*} 
Since the support of the minimizers $\bfu_{*}^{n}$ remain in a compact set $\overline{\Omega} \times \cdots \times \overline{\Omega}$, there exists a subsequence which is weak-$*$ convergent as distributions to a limit $\bfu_{*}$, which satisfies \eqref{eq:local-2phase} with $f_{i}=\mu_{i}-\lambda_{i}$. 
\end{proof}

For later convenience, we introduce the notation 
\begin{equation*}
\begin{aligned}
\min\{\bfu,\bfv\} &:= (\min\{u_{1},v_{1}\} , \cdots , \min\{u_{m},v_{m}\}), \\
\max\{\bfu,\bfv\} &:= (\max\{u_{1},v_{1}\} , \cdots , \max\{u_{m},v_{m}\}). 
\end{aligned}
\end{equation*}
We now prove Theorem~\ref{thm:support-minimizer} by modifying the ideas in \cite[Theorem~1]{AS16MultiPhaseQD}. 

\begin{proof}[Proof of Theorem~\ref{thm:support-minimizer}]
Let $\bfu_{*}=(u_{*,1},\cdots,u_{*,m}) \in \mS_{m}(\Omega)$ be a minimizer of the problem \eqref{eq:existence-global-minimizer}. Since 
\begin{equation*}
\begin{aligned}
\min\{u_{*,i},v_{*,i}\}+\max\{u_{*,i},v_{*,i}\} &= u_{*,i} + v_{*,i} \\
\abs{\min\{u_{*,i},v_{*,i}\}}^{2} + \abs{\max\{u_{*,i},v_{*,i}\}}^{2} &= \abs{u_{*,i}}^{2} + \abs{v_{*,i}}^{2} \\
\int_{\Omega} \left( \abs{\nabla \min\{u_{*,i},v_{*,i}\}}^{2} + \abs{\nabla \max\{u_{*,i},v_{*,i}\}}^{2} \right) \, \rmd x &= \int_{\Omega} \left( \abs{\nabla u_{*,i}}^{2} + \abs{\nabla v_{*,i}}^{2} \right) \, \rmd x
\end{aligned}
\end{equation*}
and $k_{1}=k_{2}=\cdots=k_{m}\in[0,k_{*})$, then 
\begin{equation*}
\mJ_{\bfk}(\min\{\bfu_{*},\bfv_{*}\}) + \mJ_{\bfk}(\max\{\bfu_{*},\bfv_{*}\}) = \mJ_{\bfk}(\bfu_{*}) + \mJ_{\bfk}(\bfv_{*}).
\end{equation*}
Since  $\min\{\bfu_{*},\bfv_{*}\}\in\mS_{m}(\Omega)$, then 
\begin{equation*}
\mJ_{\bfk}(\bfu_{*}) \le \mJ_{\bfk}(\min\{\bfu_{*},\bfv_{*}\}).
\end{equation*}
Hence we reach 
\begin{equation*}
\mJ_{\bfk}(\max\{\bfu_{*},\bfv_{*}\}) \le \mJ_{\bfk}(\bfv_{*}).
\end{equation*}
Since $\max\{\bfu_{*},\bfv_{*}\} \in \mK_{m}(\Omega)$ and $\bfv_{*}$ is the unique minimizer of $\mJ_{\bfk}$ in $\mK_{m}(\Omega)$, then 
\begin{equation*}
\max\{\bfu_{*},\bfv_{*}\} = \bfv_{*},
\end{equation*}
which proves our lemma. 
\end{proof}

We are now in position to prove Theorem~\ref{thm:m-phase-quadrature-domains} by modifying the ideas in \cite[Theorem~7]{AS16MultiPhaseQD} or \cite[Theorem~5.1]{EPS11TwoPhaseQD}. 

\begin{proof}[Proof of Theorem~\ref{thm:m-phase-quadrature-domains}]
Let $\bfu_{*}=(u_{*,1},\cdots,u_{*,m})$ be a minimizer of $\mJ_{\bfk}$ in $\mS_{m}(\Omega)$ with $\bfk=(k,k,\cdots,k)$. The remaining task is to prove the support condition \eqref{eq:support-condition-1}. In order to do this, we only need to show 
\begin{equation*}
U_{i} \subset \supp\,(u_{*,i}). 
\end{equation*}

Suppose the contrary, assuming the existence of $i_{0}$ such that $U_{i_{0}} \setminus \supp\,(u_{*,i_{0}}) \neq \emptyset$. Let's fix $z_{0} \in U_{i_{0}} \setminus \supp\,(u_{*,i_{0}})$. According to Theorem~\ref{thm:support-minimizer}, we have $\supp\,(u_{*,i}) \subset \overline{Q_{i}}$ for all $i=1,\ldots,m$. Consequently, using \eqref{eq:disjoint-condition}, it is evident that $z_{0} \in U_{i_{0}} \setminus \tilde{\Omega}$, where $\tilde{\Omega}=\bigcup_{i=1}^{m} \supp\,(u_{*,i})$. Since $\tilde{\Omega}$ is compact and $U_{i_{0}}$ is open, we can find $0<R<\beta k^{-1}$ with $0 < \beta < j_{\frac{n-2}{2},1}$ such that $\overline{B_{R}(z_{0})} \cap \tilde{\Omega} =\emptyset$ and $B_{R}(z_{0})\subset U_{i_{0}}$.

Let $0 < \varepsilon < M$ be such that 
\begin{equation*}
\max_{i=1,\cdots,m} \norm{\mu_{i}}_{L^{\infty}(\mR^{n})} \le M ,\quad \min_{i=1,\cdots,m} \lambda_{i} \ge \varepsilon.
\end{equation*}
Let $0<r<R$ be a constant to be determined later, and we define 
\begin{equation*}
\nu_{i_{0}} := \mu_{i_{0}} \chi_{B_{r}(z_0)}. 
\end{equation*}
It is easy to see that 
\begin{equation*}
a_{0} \chi_{B_{r}(z_0)} \le \nu_{i_{0}} \le M\chi_{B_{r}(z_0)} ,\quad a_{0} := \inf_{x \in \supp\,(\mu_{i})}\mu_{i}(x).
\end{equation*}
Since the mapping $t \mapsto t^{\frac{n}{2}} J_{\frac{n}{2}}(kt)$ is monotone increasing on $(0,\beta k^{-1})$, we can choose $r>0$ sufficiently small so that 
\begin{equation}
\frac{\varepsilon}{M} > \frac{r^{\frac{n}{2}} J_{\frac{n}{2}}(kr)}{R^{\frac{n}{2}} J_{\frac{n}{2}}(kR)} \label{eq:choice-r}
\end{equation}
and hence 
\begin{equation*}
R' := \max \left\{ \rho \in (r,\beta k^{-1}] : \frac{\varepsilon}{M} - \frac{r^{\frac{n}{2}} J_{\frac{n}{2}}(kr)}{\rho^{\frac{n}{2}} J_{\frac{n}{2}}(k\rho)} \le 0 \right\} < R.
\end{equation*}
The sufficient condition of \cite[Proposition~7.4]{KSS23Minimization} can be verified by \eqref{eq:non-degenerate-condition} and \eqref{eq:choice-r}, and one sees that 
\begin{equation}
\supp\,(\tilde{v}_{*,i_{0}}) \subset \overline{B_{R'}(z_{0})} \subset B_{R}(z_{0}) ,\quad B_{r_{0}} \subset \supp\,(\tilde{v}_{*,i_{0}}), \label{eq:1-phase-QD}
\end{equation}
where $\tilde{v}_{*,i_{0}}$ is the unique minimizer to the functional 
\begin{equation*}
v \mapsto \int_{B_{\beta k^{-1}}(z_{0})} \left( \abs{\nabla v}^{2} - k^{2}\abs{v}^{2} - 2(\nu_{i_{0}}-\lambda_{i_{0}})v \right) \, \rmd x 
\end{equation*}
in $\mK_{1}(B_{\beta k^{-1}}(z_{0}))$.
Since $B_{R}(z_{0}) \subset \Omega$, then in particular $\tilde{v}_{*,i_{0}}$ is also the unique minimizer to the functional 
\begin{equation*}
\mJ_{k,\nu_{i_{0}},\lambda_{i_{0}}}(v) = \int_{\Omega} \left( \abs{\nabla v}^{2} - k^{2}\abs{v}^{2} - 2(\nu_{i_{0}}-\lambda_{i_{0}})v \right) \, \rmd x
\end{equation*}
in $\mK_{1}(\Omega)$. The second inclusion in \eqref{eq:1-phase-QD} implies $\{\tilde{v}_{*,i_{0}}>0\}\neq\emptyset$, therefore $\mJ_{k,\nu_{i_{0}},\lambda_{i_{0}}}(\tilde{v}_{*,i_{0}})<0$. 

Since $\overline{B_{R}(z_{0})}\cap\tilde{\Omega}=\emptyset$, then first inclusion in \eqref{eq:1-phase-QD} implies $\supp\,(\tilde{v}_{*,i_{0}})\cap\tilde{\Omega}=\emptyset$, therefore 
\begin{equation*}
\bfw_{*} = (u_{*,1},\cdots,u_{*,i_{0}-1},u_{*,i_{0}}+\tilde{v}_{*,i_{0}},u_{*,i_{0}+1},\cdots,u_{*,m}) \in \mS_{m}(\Omega), 
\end{equation*}
as well as $\nabla \tilde{v}_{*,i_{0}} \cdot\nabla u_{*,i_{0}}=0$ and $\tilde{v}_{*,i_{0}} u_{*,i_{0}}=0$. If we consider the functional $\mJ_{\bfk}$ with $\bfk=(k,\cdots,k)$ and $f_{i} = \mu_{i}-\lambda_{i}$, then we see that  
\begin{equation*}
\begin{aligned}
\mJ_{\bfk}(\bfw_{*}) & = \sum_{i\neq i_{0}} \int_{\Omega} \left( \abs{\nabla u_{*,i}}^{2} - k^{2}\abs{u_{*,i}}^{2} - 2(\mu_{i}-\lambda_{i})u_{*,i} \right) \, \rmd x \\
& \quad + \int_{\Omega} \left( \abs{\nabla (u_{*,i_{0}}+\tilde{v}_{*,i_{0}})}^{2} - k^{2}\abs{u_{*,i_{0}}+\tilde{v}_{*,i_{0}}}^{2} - 2(\mu_{i_{0}}-\lambda_{i_{0}})(u_{*,i_{0}}+\tilde{v}_{*,i_{0}}) \right) \, \rmd x \\
& = \mJ_{k,\cdots,k}(\bfu_{*}) + \int_{\Omega} \left( \abs{\nabla \tilde{v}_{*,i_{0}}}^{2} - k^{2}\abs{\tilde{v}_{*,i_{0}}}^{2} - 2(\mu_{i_{0}}-\lambda_{i_{0}})\tilde{v}_{*,i_{0}} \right) \, \rmd x \\
& = \mJ_{k,\cdots,k}(\bfu_{*}) + \mJ_{k,\nu_{i_{0}},\lambda_{i_{0}}}(\tilde{v}_{*,i_{0}}) < \mJ_{k,\cdots,k}(\bfu_{*}),
\end{aligned}
\end{equation*}
which contradicts the minimality of $\bfu_{*} \in \mS_{m}(\Omega)$. 

Therefore, we conclude that $U_{i}\setminus\supp\,(u_{*,i})=\emptyset$ for all $i=1,\cdots,m$, which is equivalent to $U_{i} \subset \supp\,(u_{*,i})$ for all $i=1,\cdots,m$, which concludes our theorem. 
\end{proof}

\section{Two-phase problem through partial balayage}

\addtocontents{toc}{\SkipTocEntry}
\subsection{Main results}

It appears that Theorem~\ref{thm:m-phase-quadrature-domains} does not ensure the crucial support condition \eqref{eq:support-condition-2-phase} in our application. Even for the one-phase case with $k=0$ there is no simple way to guarantee the support condition. To address this limitation, we slightly refine Theorem~\ref{thm:m-phase-quadrature-domains} specifically for the case of $m=2$, as presented in Theorem~\ref{thm:construction-2-phase-partial-balayage} below. This refinement employs a potential-theoretic analysis known as \emph{partial balayage} \cite{GS09PartialBalayage,GS24PartialBalayageHelmholtz,Gus04LecturesBalayage,GR18PartialBalayageManifold,GS94PartialBayage}. The framework adopted here largely follows the concepts outlined in \cite{GS12TwoPhaseQD,GS24PartialBalayageHelmholtz}. For ease of discussion, we introduce the same terminologies as in
\cite{GS24PartialBalayageHelmholtz}.\footnote{The approach based on Balayage, does not straightforwardly generalize to multi-phase cases, and one would need to find an enhanced version of it to assure an existence theory with the support-inclusion  conditions \eqref{eq:support-condition-2-phase}. }

\begin{def*}
Let $k > 0$. A function $s$ that is upper semicontinuous (USC) and satisfies $(\Delta + k^{2})s \ge 0$ in the sense of distributions will be referred to as \emph{$k$-metasubharmonic}.

Similarly, $s$ will be termed 
\emph{$k$-metasuperharmonic} if $-s$ is $k$-metasubharmonic. Additionally, we use the term \emph{$k$-metaharmonic} when $s$ is both $k$-metasubharmonic and $k$-metasuperharmonic.
\end{def*}

The partial balayage heavily relies on the following concept: 

\begin{def*}
We say that the \emph{$k$-maximum principle} holds on a domain (i.e. open and connected) $\Omega \subset \mR^{n}$ if the following properties holds: Every $k$-metasubharmonic function $s$ which is bounded from above and satisfies 
\begin{equation*}
\limsup_{x \rightarrow z} s(x) \le 0 \text{ for all $z\in\partial\Omega$ apart (possibly) from a polar set}
\end{equation*}
must also satisfy $s\le 0$ in $\Omega$. 
\end{def*}

We introduce the number
\begin{equation*}
 \tilde{k}_{*}(\Omega) := \sup \left\{ k : \text{there exists USC $s<0$ in $\Omega$ satisfying $(\Delta + k^{2})s \ge 0$ in $\Omega$} \right\}.
 \end{equation*}
It is worth mentioning that there exists a positive eigenfunction $h$ of $-\Delta$ corresponding to $\tilde{k}_{*}=\tilde{k}_{*}(\Omega)$, meaning
\begin{equation*}
(\Delta + \tilde{k}_{*})h=0 \text{ and } h>0 \text{ in $\Omega$} ,\quad \left. h \right|_{\partial\Omega} = 0 \text{ in a suitable sense},
\end{equation*}
as stated more precisely in \cite[Theorem~2.1]{BNV94MaximumPrinciple}. Importantly, it is also noted that $\tilde{k}_{*}(\Omega)=k_{*}(\Omega)$, where $k_{*}$ is the number given in \eqref{eq:principal-eigenvalue}, a result that holds true for \emph{arbitrary} bounded domains $\Omega$, as shown in \cite[Proposition~2.6]{GS24PartialBalayageHelmholtz}. The connectedness of $\Omega$ is crucial here. Additionally, it was demonstrated in \cite[Proposition~2.8]{GS24PartialBalayageHelmholtz} that the following are equivalent 
(see also \cite[Theorem~1.1]{BNV94MaximumPrinciple}): 
\begin{enumerate}
\renewcommand{\labelenumi}{\theenumi}
\renewcommand{\theenumi}{{\rm (\roman{enumi})}}    
\item \label{itm:maximum-equivalence-i} the $k$-maximum principle holds on $\Omega$; 
\item $0 \le k < k_{*}$; 
\item there is a positive $k$-metasuperharmonic function on $\Omega$ which is not a multiple of $h$; 
\item \label{itm:maximum-equivalence-iv} there is a $k$-metaharmonic function $v\ge 1$ in $\Omega$. 
\end{enumerate}

By imitating some ideas in \cite{GS09PartialBalayage,GS12TwoPhaseQD}, we now introduce a version of partial balayage which is slightly general than the one in \cite[Section~3]{GS24PartialBalayageHelmholtz}. Let $k>0$. Given an open set $D \subset \mR^{n}$ and a positive measure $\mu$ with compact support in $\mR^{n}$, we define 
\begin{equation*}
\mF_{k,D}(\mu) := \left\{ v\in\mD'(\mR^{n}) : \begin{aligned}
& \text{$-(\Delta + k^{2})v \le 1$ in $D$} ,\quad \text{$v \le U_{k}^{\mu}$ in $\mR^{n}$} \\
& \text{the set $\{v < U_{k}^{\mu}\}$ is bounded}
\end{aligned} \right\}, 
\end{equation*}
where the potential is given by 
\begin{equation*}
U_{k}^{\mu}(x) := \int \Psi_{k}(x-y)\,\rmd\mu(y) ,\quad \Psi_{k}(x) = -\frac{1}{4} \left(\frac{k}{2\pi}\right)^{\frac{n-2}{2}} \abs{x}^{-\frac{n-2}{2}} Y_{\frac{n-2}{2}}(k\abs{x}). 
\end{equation*}
We simply denote 
\begin{equation*}
\mF_{k}(\mu) := \mF_{k,\mR^{n}}(\mu). 
\end{equation*}

Obviously $\mF_{k}(\mu) \subset \mF_{k,D}(\mu)$, and by \cite[Theorem~1.4]{GS24PartialBalayageHelmholtz}, one can guarantee $\mF_{k}(\mu) \neq \emptyset$, and so is $\mF_{k,D}(\mu)$, when $k>0$ is sufficiently small such that 
\begin{subequations}
\begin{equation}
\mu(\mR^{n}) \le c_{k}(R_{k}), \label{eq:non-empty-condition}
\end{equation}
where 
\begin{equation*}
R_{k} := k^{-1} j_{\frac{n-2}{2}} ,\quad c_{k}(r) = \left(\frac{2\pi r}{k}\right)^{\frac{n}{2}} J_{\frac{n}{2}}(kr) = \int_{0}^{r} (2\pi r)^{\frac{n}{2}} k^{-\frac{n-2}{2}} J_{\frac{n-2}{2}}(kr),
\end{equation*}
see also \cite[Corollary~3.21 and Corollary~3.22]{GS24PartialBalayageHelmholtz} for some refinements.

By using the ideas in \cite[Lemma~2.4]{GS24PartialBalayageHelmholtz}, which involves   Kato's inequality for the Laplacian \cite{BP04KatoInequality} (see also \cite[Corollary~2.3]{GS12TwoPhaseQD}), one sees that 
\begin{equation*}
\max\{u,v\} \in \mF_{k,D}(\mu) \quad \text{for all $u,v\in \mF_{k,D}(\mu)$.}
\end{equation*}
Standard potential theoretic arguments \cite[Section~3.7]{AG01Potential} now show that $\mF_{k,D}(\mu)$ has a largest element, which has a USC representative. We denote this function by $V_{k,D}^{\mu}$, which also can be referred to  as the \emph{partial reduction} of $U_{k}^{\mu}$ \cite{GS09PartialBalayage}. Accordingly, we can define the \emph{non-contact set} by 
\begin{equation*}
\omega_{k,D}(\mu) := \left\{ V_{k,D}^{\mu} < U_{k}^{\mu} \right\} ,\quad \omega_{k}(\mu) := \omega_{k,\mR^{n}}(\mu),
\end{equation*}
and the \emph{partial balayage} is defined by 
\begin{equation*}
\Bal_{k,D}(\mu) := -(\Delta + k^{2})V_{k,D}^{\mu} \text{ in $\mD'(\mR^{n})$} ,\quad \Bal_{k}(\mu):=\Bal_{k,\mR^{n}}(\mu).
\end{equation*}
Obviously one has $V_{k}^{\mu} \le V_{k,D}^{\mu}$ and $\omega_{k,D}(\mu) \subset \omega_{k}(\mu) \cap D$ for any open set $D$. 
If we further assume that 
\begin{equation}
\mu(\mR^{n}) < c_{k}(R_{k}), \label{eq:non-empty-condition-stronger}
\end{equation}
since $\omega_{k}(\mu)$ is bounded, then one can choose $\epsilon>0$ such that 
\begin{equation*}
(\mu + \epsilon \sfm|_{\omega_{k}(\mu)})(\mR^{n}) = \mu(\mR^{n}) + \epsilon\sfm(\omega_{k}(\mu)) \le c_{k}(R_{k}). 
\end{equation*}
again using \cite[Theorem~1.4]{GS24PartialBalayageHelmholtz}, we see that $\mF_{k}(\mu + \epsilon\sfm|_{\omega_{k}(\mu)}) \neq \emptyset$. Consequently, by using \cite[Theorem~3.9]{GS24PartialBalayageHelmholtz} we see that 
\begin{equation}
\text{$\omega_{k}(\mu)$ satisfies the $k$-maximum principle \ref{itm:maximum-equivalence-i}--\ref{itm:maximum-equivalence-iv}.} \label{eq:maximum-principle-noncontact-set}
\end{equation}
\end{subequations}
By using the fact $V_{k,D}^{\mu}=U_{k}^{\mu}$ in $\mR^{n}\setminus D$, one can easily verify that 
\begin{equation*}
V_{k,D}^{\mu} = U_{k}^{\Bal_{k,D}(\mu)}.
\end{equation*}
Following the  same arguments as in \cite{GS09PartialBalayage,GS24PartialBalayageHelmholtz}, similar to \cite[(5)]{GS12TwoPhaseQD} or \cite[(4)]{SS13TwoPhase}, one also can show that there exists a measure $\nu \ge 0$ which is supported on $\partial D \cap \partial \omega_{k,D}(\mu)$ such that 
\begin{equation}
\Bal_{k,D}(\mu) = \left. \sfm \right|_{\omega_{k,D}(\mu)} + \left. \mu \right|_{\mR^{n}\setminus\omega_{k,D}(\mu)} + \nu, \label{eq:structure-balayage}
\end{equation}
where $\sfm$ is the usual Lebesgue measure. In addition, $\Bal_{k,D}(\mu) \le 1$ in $D$. Here and after, we identify $\sfm$ with 1. When $D=\mR^{n}$, the above definitions are identical to the one mentioned in \cite{GS24PartialBalayageHelmholtz}.

We will prove the following analogue to \cite[Theorem~5.1]{GS12TwoPhaseQD}. 

\begin{thm}[See \eqref{eq:representation-2-phase} below for a more precise description]\label{thm:construction-2-phase-partial-balayage}
Let $\mu_{\pm}$ be positive measures with disjoint compact supports in $\mR^{n}$, and let $k>0$ satisfies
\begin{equation}
\mu_{+}(\mR^{n}) + \mu_{-}(\mR^{n}) < c_{k}(R_{k}) \equiv (2\pi j_{\frac{n-2}{2}})^{\frac{n}{2}} J_{\frac{n}{2}}(j_{\frac{n-2}{2}}) k^{-n} \label{eq:k-small-condition}
\end{equation}
If 
\begin{equation}
\overline{\omega_{k}(\mu_{-})} \cap \supp\,(\mu_{+}) = \emptyset ,\quad \overline{\omega_{k}(\mu_{+})} \cap \supp\,(\mu_{-}) = \emptyset \label{eq:disjoint-condition-partial-balayage}
\end{equation}
and additionally assume that 
\begin{equation}
\supp\,(\mu_{+}) \subset \omega_{k,\mR^{n}\setminus\overline{\omega(\mu_{-})}} (\mu_{+}) ,\quad \supp\,(\mu_{-}) \subset \omega_{k,\mR^{n}\setminus\overline{\omega(\mu_{+})}} (\mu_{-}), \label{eq:extra-support-condition} 
\end{equation}
then there exist two disjoint open bounded sets $D_{\pm}$ such that $(D_{+},D_{-})$ is a two-phase $(k,k)$-quadrature domain, in the sense of \eqref{eq:main-2-phase}, with $\lambda_{+}=\lambda_{-}=1$, which the support condition \eqref{eq:support-condition-2-phase} holds. 
\end{thm}

By combining Theorem~\ref{thm:construction-2-phase-partial-balayage} and \cite[Theorem~7.1 and Remark~7.2]{KLSS22QuadratureDomain}, we also can prove the following result: 

\begin{thm}\label{thm:construction-2-phase-partial-balayage-concentration}
Let $\mu_{\pm}$ be positive measures with disjoint compact supports in $\mR^{n}$. There exists a positive constant $c_{n}$ depending only on dimension $n$ such that the following statement holds true: If $k>0$ satisfies
\begin{equation}
0 < k < \frac{c_{n}}{(\mu_{+}+\mu_{-})(\mR^{n})^{1/n}} \label{eq:smallness-wavenumber-k}
\end{equation}
and $\mu_{\pm}$ satisfy \eqref{eq:disjoint-condition-partial-balayage} as well as the  concentration condition 
\begin{equation}
\begin{aligned} 
& \limsup_{r\rightarrow 0_{+}} \frac{\mu_{+}(B_{r}(x))}{r^{n}} > \frac{1}{c_{n}} \quad \text{for all $x\in\supp\,(\mu_{+})$,} \\
& \limsup_{r\rightarrow 0_{+}} \frac{\mu_{-}(B_{r}(y))}{r^{n}} > \frac{1}{c_{n}} \quad \text{for all $y\in\supp\,(\mu_{-})$,}
\end{aligned} \label{eq:concentration-condition}
\end{equation}
then there exist two disjoint open bounded sets $D_{\pm}$ such that $(D_{+},D_{-})$ is a two-phase $(k,k)$-quadrature domain, in the sense of \eqref{eq:main-2-phase}, with $\lambda_{+}=\lambda_{-}=1$, which the support condition \eqref{eq:support-condition-2-phase} holds. 
\end{thm}

\addtocontents{toc}{\SkipTocEntry}
\subsection{Proofs of the theorems}

Given a signed measure $\mu=\mu_{+}-\mu_{-}$ with compact support and a Borel function $u:\mR^{n}\rightarrow [-\infty,+\infty]$, we define the signed measure 
\begin{equation*}
\eta(u,\mu):= \left( (\mu_{+}-1)_{+}-\left.(\mu_{+}-1)_{-}\right|_{\{u>0\}} \right) - \left( (\mu_{-}-1)_{+}-\left.(\mu_{-}-1)_{-}\right|_{\{u<0\}} \right). 
\end{equation*}
We recall the properties of $\eta$: 

\begin{lem}[{\cite[Lemma~4.1]{GS12TwoPhaseQD}}] \label{lem:GS12TwoPhaseQD-lem4.1}
Let $u,u_{1},u_{2}:\mR^{n}\rightarrow[-\infty,+\infty]$ be Borel measurable functions, $\mu,\mu_{1},\mu_{2}$ be signed measures with compact supports, and $A \subset \mR^{n}$ be a Borel set. Then 
\begin{enumerate}
\renewcommand{\labelenumi}{\theenumi}
\renewcommand{\theenumi}{{\rm (\alph{enumi})}}  
\item \label{itm:GS12TwoPhaseQD-lem4.1_a}$\eta(-u,-\mu)=-\eta(u,\mu)$;  
\item $\mu-1 \le \eta(u,\mu) \le \mu + 1$; and 
\item \label{itm:GS12TwoPhaseQD-lem4.1_c} $u_{1}|_{A}\le u_{2}|_{A}$ and $\mu_{1}|_{A} \ge \mu_{2}|_{A}$ imply $\eta(u_{1},\mu_{1})|_{A} \ge \eta(u_{2},\mu_{2})|_{A}$. 
\end{enumerate}
\end{lem}

Similar to \cite[Section~2.2]{GS12TwoPhaseQD}, by a \emph{$\delta$-$k$-metasubharmonic} function on an open set $\Omega$ we mean a function $w = s_{1}-s_{2}$ for some $k$-metasubharmonic functions $s_{1}$ and $s_{2}$ on $\Omega$, which is well-defined outside the polar set where $s_{1}=s_{2}=-\infty$. By using exactly same ideas there, we also can refine this observation using the fine topology: As a distribution, $-(\Delta + k^{2})w$ is locally a signed measure, and there exists a unique decomposition 
\begin{equation*}
-(\Delta + k^{2})w = (-(\Delta + k^{2})w)_{\rm d} + (-(\Delta + k^{2})w)_{\rm c}
\end{equation*}
where $(-(\Delta + k^{2})w)_{\rm d}$ does not charge polar sets and $(-(\Delta + k^{2})w)_{\rm c}$ is carried by a polar set. We always assign values to a $\delta$-$k$-metasubharmonic function in the following way (without explicitly mention after that): 
\begin{equation*}
\begin{aligned} 
&\text{$w := +\infty$ a.e. with respect to $((-(\Delta + k^{2})w)_{\rm c})_{+}$}, \\
& \text{$w := -\infty$ a.e. with respect to $((-(\Delta + k^{2})w)_{\rm c})_{-}$}, 
\end{aligned} 
\end{equation*}
where $(-(\Delta + k^{2})w)_{\rm c} = ((-(\Delta + k^{2})w)_{\rm c})_{+} - ((-(\Delta + k^{2})w)_{\rm c})_{-}$ is the Jordan decomposition of $(-(\Delta + k^{2})w)_{\rm c}$. 

It is convenient to define $W_{k,D}^{\mu} := U_{k}^{\mu} - V_{k,D}^{\mu}$, whence $W_{k,D}^{\mu}$ is lower semicontinuous (we use the abbreviation ``LSC''), and we also denote $W_{k}^{\mu} := W_{k,\mR^{n}}^{\mu}$. We now define 
\begin{equation*}
\tau_{k,\mu} := \left\{ w : \text{$w$ is $\delta$-$k$-metasubharmonic, $-(\Delta + k^{2})w \ge \eta(w,\mu)$ and $w \ge -W_{k}^{\mu_{-}}$ in $\mR^{n}$} \right\}.
\end{equation*}
Fix any $\varphi \in C^{\infty}(\mR^{n})$ with $(\Delta+k^{2})\varphi=1$, we also can consider the collection 
\begin{equation*}
\tau_{k,\mu}' := \left\{ w + U_{k}^{\mu_{-}} - \varphi : w \in \tau_{k,\mu} \right\}, 
\end{equation*}
where the elements of $\tau_{k,\mu}'$ are suitably refined on a polar set to make them $k$-metasuperharmonic. When $k=0$, one can simply choose $\varphi(x)=\abs{x}^{2}/2n$. We now modify \cite[Lemma~4.2]{GS12TwoPhaseQD} in the following lemma: 

\begin{lem}\label{lem:GS12TwoPhaseQD-lem4.2}
Let $\mu_{\pm}$ be positive measures with disjoint compact supports in $\mR^{n}$, and let $k>0$ satisfies \eqref{eq:non-empty-condition} with respect to $\mu_{\pm}$. If $v_{1},v_{2} \in \tau_{k,\mu}'$ with $\mu=\mu_{+}-\mu_{-}$, then $\min\{v_{1},v_{2}\} \in \tau_{k,\mu}'$. 
\end{lem}

\begin{proof}
Let $v_{1},v_{2} \in \tau_{k,\mu}'$ and write $v_{i} = w_{i} + U_{k}^{\mu_{-}} - \varphi$ where $w_{i} \in \tau_{k,\mu}$. Following the arguments in \cite[Lemma~4.2]{GS12TwoPhaseQD}, by using \cite[Lemma~2.4]{GS24PartialBalayageHelmholtz} one can show that $\min\{v_{1},v_{2}\}$ is $\delta$-$k$-metasubharmonic function and $\min\{w_{1},w_{2}\} \ge -W_{k}^{\mu_{-}}$ in $\mR^{n}$, as well as 
\begin{equation*}
\eta(\min\{w_{1},w_{2}\},\mu) = \eta(w_{1},\mu)|_{\{w_{1}-w_{2} \le 0\}} + \eta(w_{2},\mu)|_{\{w_{1}-w_{2} > 0\}}. 
\end{equation*}
By using Kato's inequality for Laplacian, one further computes that 
\begin{equation*}
\begin{aligned} 
& \eta(\min\{w_{1},w_{2}\},\mu) \le -(\Delta+k^{2})w_{1}|_{\{w_{1}-w_{2} \le 0\}} -(\Delta+k^{2})w_{2}|_{\{w_{1}-w_{2} > 0\}} \\
& \quad \le -\Delta \min\{w_{1},w_{2}\} - k^{2}w_{1}|_{\{w_{1}\le w_{2}\}} - k^{2}w_{2}|_{\{w_{2}<w_{1}\}} = -(\Delta + k^{2})\min\{w_{1},w_{2}\},
\end{aligned} 
\end{equation*}
we conclude our lemma. 
\end{proof}

We can prove the following two technical lemmas by employing the ideas presented in \cite[Theorem~4.3]{GS12TwoPhaseQD}.

\begin{lem}\label{lem:GS12TwoPhaseQD-Thm4.3a}
Let $\mu_{\pm}$ be positive measures with disjoint compact supports in $\mR^{n}$, and let $k>0$. Consider $u_{1}$ and $u_{2}$ as $\delta$-$k$-metasubharmonic functions with 
\begin{equation*}
\{u_{1}\neq 0\} \cup \{u_{2}\neq 0\} \subset \tilde{\Omega} 
\end{equation*}
for some open set $\tilde{\Omega}$ such that the $k$-maximum principle \ref{itm:maximum-equivalence-i}--\ref{itm:maximum-equivalence-iv} holds. 
If $-(\Delta+k^{2}) u_{1} \ge \eta(u_{1},\mu)$ and $-(\Delta + k^{2})u_{2} \le \eta(u_{2},\mu)$ with $\mu=\mu_{+}-\mu_{-}$, then it follows that $u_{2} \le u_{1}$. 
\end{lem}

\begin{proof} 
One computes that the function $v = u_{2} - u_{1}$ satisfies 
\begin{equation*}
\begin{aligned} 
& -(\Delta + k^{2})v \le \eta(u_{2},\mu) - \eta(u_{1},\mu) \\ 
& \quad = \left. (\mu_{+}-1)_{-} \right|_{\{u_{1}>0\}} - \left. (\mu_{+}-1)_{-} \right|_{\{u_{2}>0\}} + \left. (\mu_{-}-1)_{-}\right|_{\{u_{2}<0\}} - \left. (\mu_{-}-1)_{-}\right|_{\{u_{1}<0\}}, 
\end{aligned} 
\end{equation*}
so $\left. -(\Delta + k^{2})v \right|_{\{v \ge 0\}} \le 0$. By using the Kato's inequality as in \cite[Lemma~2.4]{GS24PartialBalayageHelmholtz}, one sees that 
\begin{equation*}
(\Delta + k^{2})v_{+} \ge \left. (\Delta + k^{2})v \right|_{\{v \ge 0\}} \ge 0, 
\end{equation*}
in other words, $v_{+}$ is $k$-metasubharmonic. Since $\{v_{+}>0\} \subset \tilde{\Omega}$, we conclude $v_{+} \equiv 0$ by $k$-maximum principle \ref{itm:maximum-equivalence-i}--\ref{itm:maximum-equivalence-iv}, which completes our proof. 
\end{proof}

If we have the assumption \eqref{eq:k-small-condition}, from the discussions in \eqref{eq:maximum-principle-noncontact-set} above we know that $\left.\mF_{k}(\mu_{+}+\mu_{-})\right.\neq\emptyset$ and 
\begin{subequations}
\begin{equation}
\text{$\omega_{k}(\mu_{+}+\mu_{-})$ satisfies the $k$-maximum principle \ref{itm:maximum-equivalence-i}--\ref{itm:maximum-equivalence-iv}}. \label{eq:k-maximum-principle-assump}
\end{equation}
By using \cite[Lemma~3.3]{GS24PartialBalayageHelmholtz}, we have 
\begin{equation}
\omega_{k}(\mu_{+}) \cup \omega_{k}(\mu_{-}) \subset \omega_{k}(\mu_{+}+\mu_{-}). \label{relation-k-maximum-domains}
\end{equation}
\end{subequations}
Based on these observations, we now able to proof the following lemma.

\begin{lem}\label{lem:GS12TwoPhaseQD-Thm4.3b}
Let $\mu_{\pm}$ be positive measures with disjoint compact supports in $\mR^{n}$, and let $k>0$ satisfy \eqref{eq:k-small-condition}. Consider $u$ as a $\delta$-$k$-metasubharmonic function with $\{u\neq 0\} \subset \omega_{k}(\mu_{+} + \mu_{-})$. Then the following hold:
\begin{enumerate}
\renewcommand{\labelenumi}{\theenumi}
\renewcommand{\theenumi}{{\rm (\alph{enumi})}} 
\item \label{itm:lem:GS12TwoPhaseQD-Thm4.3b-a}If $-(\Delta + k^{2})u \le \eta(u,\mu)$ with $\mu=\mu_{+}-\mu_{-}$, then $u \le W_{k}^{\mu_{+}}$. 
\item \label{itm:lem:GS12TwoPhaseQD-Thm4.3b-b}If $-(\Delta + k^{2})u \ge \eta(u,\mu)$ with $\mu=\mu_{+}-\mu_{-}$, then $u \ge -W_{k}^{\mu_{-}}$ and so $u \in \tau_{k,\mu}$. 
\end{enumerate}
\end{lem}

\begin{proof}
First of all, we remind the readers that $W_{k}^{\mu_{+}}$ is non-negative (see the definition of $V_{k}^{\mu_{+}}$ and the definition of $\mF_{k}(\mu_{+})$), $\delta$-$k$-metasubharmonic and has compact support. Since $\Bal_{k}(\mu_{+}) \le 1$ in $\mR^{n}$, by the structure of partial balayage \eqref{eq:structure-balayage} we see that   
\begin{equation*}
\left. \mu_{+} \right|_{{\{W_{k}^{\mu_{+}}}=0\}} = \left. \mu_{+} \right|_{\mR^{n}\setminus\omega_{k}(\mu_{+})} \le 1. 
\end{equation*}
Consequently, together with Lemma~\ref{lem:GS12TwoPhaseQD-lem4.1}\ref{itm:GS12TwoPhaseQD-lem4.1_c} we compute that 
\begin{equation*}
\begin{aligned}
& -(\Delta + k^{2}) W_{k}^{\mu_{+}} = \mu_{+} - \Bal_{k}(\mu_{+}) = \mu_{+} - \sfm|_{\{W_{k}^{\mu_{+}}>0\}} - \left. \mu_{+} \right|_{\{W_{k}^{\mu_{+}}=0\}} \\
& \quad = \left. (\mu_{+} - 1) \right|_{\{W_{k}^{\mu_{+}}>0\}} = (\mu_{+} - 1)_{+} - \left. (\mu_{+} - 1)_{-} \right|_{\{W_{k}^{\mu_{+}}>0\}} = \eta(W_{k}^{\mu_{+}},\mu_{+}) \\ 
& \quad \ge \eta(W_{k}^{\mu_{+}},\mu). 
\end{aligned}
\end{equation*}
Now we choose $u_{1} = W_{k}^{\mu_{+}}$, $u_{2}=u$ and $\tilde{\Omega} = \omega_{k}(\mu_{+}+\mu_{-})$ in Lemma~\ref{lem:GS12TwoPhaseQD-Thm4.3a} (together with \eqref{eq:k-maximum-principle-assump}--\eqref{relation-k-maximum-domains}) to conclude $u \le W_{k}^{\mu_{+}}$ and we complete the proof of Lemma~\ref{lem:GS12TwoPhaseQD-Thm4.3b}\ref{itm:lem:GS12TwoPhaseQD-Thm4.3b-a}.

We now replacing $\mu$ with $-\mu$ to obtain 
\begin{equation*}
\begin{aligned} 
& -(\Delta + k^{2}) (-W_{k}^{\mu_{-}}) = (\Delta + k^{2})W_{k}^{(-\mu)_{+}} = \eta(W_{k}^{(-\mu)_{+}},(-\mu)_{+}) \\ 
& \quad \le -\eta(W_{k}^{\mu_{-}},-\mu) = \eta(-W_{k}^{\mu_{-}},\mu), 
\end{aligned} 
\end{equation*}
where the last equality follows from Lemma~\ref{lem:GS12TwoPhaseQD-lem4.1}\ref{itm:GS12TwoPhaseQD-lem4.1_a}. Now we choose $u_{1}=u$, $u_{2} = - W_{k}^{\mu_{-}}$ and $\tilde{\Omega} = \omega_{k}(\mu_{+}+\mu_{-})$ in Lemma~\ref{lem:GS12TwoPhaseQD-Thm4.3a} (together with \eqref{eq:k-maximum-principle-assump}--\eqref{relation-k-maximum-domains}) to conclude $- W_{k}^{\mu_{-}} \le u$ and we complete the proof of Lemma~\ref{lem:GS12TwoPhaseQD-Thm4.3b}\ref{itm:lem:GS12TwoPhaseQD-Thm4.3b-b}.
\end{proof}

We now follow the arguments in \cite[Theorem~4.4, Theorem~4.5, Corollary~4.6, and Remark~1]{GS12TwoPhaseQD} to establish the following lemma.

\begin{lem}\label{lem:construction-2phase}
Let $\mu_{\pm}$ be positive measures with disjoint compact supports in $\mR^{n}$, and let $k>0$ satisfies \eqref{eq:k-small-condition}. If we write $\mu=\mu_{+}-\mu_{-}$, then the set $\tau_{k,\mu}$ contains a least element $\overline{W}_{k}^{\mu}$ with $\{\overline{W}_{k}^{\mu} \neq 0\} \subset \omega_{k}(\mu_{+}+\mu_{-})$\footnote{Since $W_{k}^{\mu_{+}} \ge 0 \ge -W_{k}^{\mu_{-}}$, then $W_{k}^{\mu_{+}} \in \tau_{k,\mu}$. By definition of $\tau_{k,\mu}$ and the minimality of $\overline{W}_{k}^{\mu}$, one has $W_{k}^{\mu_{+}} \ge \overline{W}_{k}^{\mu} \ge -W_{k}^{\mu_{-}}$. Consequently by \eqref{relation-k-maximum-domains}, one reaches $\{\overline{W}_{k}^{\mu} \neq 0\} \subset \omega_{k}(\mu_{+}+\mu_{-})$. This condition is essential when applying Lemma~\ref{lem:GS12TwoPhaseQD-Thm4.3b}.}. If the following support conditions hold: 
\begin{equation}
\supp\,(\mu_{\pm}) \subset D_{\pm} := \left\{ \pm\overline{W}_{k}^{\mu} > 0 \right\}, \label{eq:support-assumption}
\end{equation}
then both $D_{\pm}$ are open set in $\mR^{n}$ and the pair of domains $(D_{+},D_{-})$ is a two-phase $(k,k)$-quadrature domain with $\lambda_{+}=\lambda_{-}=1$, that is, $\tilde{u} := \overline{W}_{k}^{\mu}$ satisfies the model equation \eqref{eq:main-2-phase} with $k_{+}=k_{-}=k$ and $\lambda_{+}=\lambda_{-}=1$.
\end{lem}

We now ready to prove Theorem~\ref{thm:construction-2-phase-partial-balayage}.

\begin{proof}[Proof of Theorem~\ref{thm:construction-2-phase-partial-balayage}]
We define 
\begin{equation*}
u := W_{k}^{\mu_{+}} - W_{k,\mR^{n}\setminus\overline{\omega(\mu_{+})}}^{\mu_{-}} ,\quad v:= W_{k,\mR^{n}\setminus\overline{\omega(\mu_{-})}}^{\mu_{+}} - W_{k}^{\mu_{-}}, 
\end{equation*}
and using the disjoint condition between $\mu_{\pm}$ \eqref{eq:disjoint-condition-partial-balayage} we observe that 
\begin{equation}
\begin{aligned} 
& \{u<0\} = \omega_{k,\mR^{n}\setminus\overline{\omega(\mu_{+})}} (\mu_{-}) ,\quad \{u>0\} = \omega_{k}(\mu_{+}), \\
& \{v>0\} = \omega_{k,\mR^{n}\setminus\overline{\omega(\mu_{-})}} (\mu_{+}) ,\quad \{v<0\} = \omega_{k}(\mu_{-}). 
\end{aligned}\label{eq:observation-for-support-condition}
\end{equation}
Combining \eqref{eq:extra-support-condition} and \eqref{eq:observation-for-support-condition}, we reach 
\begin{equation}
\supp\,(\mu_{+}) \subset \{v>0\} ,\quad \supp\,(\mu_{-}) \subset \{u<0\}. \label{eq:extra-support-condition-rephrase}
\end{equation}
By using \eqref{relation-k-maximum-domains}, we see that all the sets mentioned in \eqref{eq:observation-for-support-condition} are contained in $\omega_{k}(\mu_{+}+\mu_{-})$, and thus 
\begin{equation}
\{u \neq 0\} \cup \{v \neq 0\} \subset \omega_{k}(\mu_{+}+\mu_{-}). \label{eq:verification-support-conditions}
\end{equation}
For readers' convenience, we recall \eqref{eq:k-maximum-principle-assump}: $\omega_{k}(\mu_{+}+\mu_{-})$ satisfies the $k$-maximum principle \ref{itm:maximum-equivalence-i}--\ref{itm:maximum-equivalence-iv}.

On the other hand, by using the structure of partial balayage \eqref{eq:structure-balayage}, from \eqref{eq:observation-for-support-condition} one sees that 
\begin{equation*}
\begin{aligned} 
& -(\Delta + k^{2})u = \mu_{+} - \Bal_{k}(\mu_{{+}}) - \mu_{-} + \Bal_{k,\mR^{n}\setminus\overline{\omega(\mu_{+})}}(\mu_{{-}}) \\
& \quad = (\mu_{+} - 1)|_{\omega_{k}(\mu_{+})} - (\mu_{-} - 1)|_{\omega_{k,\mR^{n}\setminus\overline{\omega(\mu_{+})}}(\mu_{-})} + \nu \\
& \quad \ge (\mu_{+} - 1)|_{\{u>0\}} - (\mu_{-} - 1)|_{\{u<0\}}.  
\end{aligned} 
\end{equation*}
In view of the structure of partial balayage \eqref{eq:structure-balayage} (with $D =\mR^{n}$), one observes that $\mu_{+} \le 1$ outside $\omega_{k}(\mu_{+})=\{u>0\}$, hence one sees that 
\begin{equation}
-(\Delta + k^{2})u \ge (\mu_{+} - 1)|_{\{u>0\}} - (\mu_{-} - 1)|_{\{u<0\}} = \eta(u,\mu). \label{eq:check-lem:GS12TwoPhaseQD-Thm4.3b}
\end{equation}
Combining \eqref{eq:check-lem:GS12TwoPhaseQD-Thm4.3b} and \eqref{eq:verification-support-conditions} with Lemma~\ref{lem:GS12TwoPhaseQD-Thm4.3b}, we conclude $u \in \tau_{k,\mu}$, and we reach $u \ge \overline{W}_{k}^{\mu}$. Consequently, from \eqref{eq:extra-support-condition-rephrase} we conclude that 
\begin{subequations} 
\begin{equation}
\supp\,(\mu_{-}) \subset \{u<0\} \subset \{ \overline{W}_{k}^{\mu}<0 \}. \label{eq:verify-sufficient-condition1}
\end{equation}
One can similar show that 
\begin{equation}
\supp\,(\mu_{+}) \subset \{v>0\} \subset \{ \overline{W}_{k}^{\mu}>0 \}. \label{eq:verify-sufficient-condition2}
\end{equation}
\end{subequations} 
This means that the support conditions \eqref{eq:support-assumption} are verified by \eqref{eq:verify-sufficient-condition1} and \eqref{eq:verify-sufficient-condition2}. 
Using the second part of Lemma~\ref{lem:construction-2phase},
we can   conclude our theorem defining 
\begin{equation} 
D_{\pm} = \left\{ \pm\overline{W}_{k}^{\mu} > 0 \right\} ,\quad \tilde{u} = \overline{W}_{k}^{\mu} .\label{eq:representation-2-phase}
\end{equation} 
\end{proof}

Using Theorem~\ref{thm:construction-2-phase-partial-balayage}, 
we can prove Theorem~\ref{thm:construction-2-phase-partial-balayage-concentration} following the ideas in \cite[Corollary~5.2]{GS12TwoPhaseQD}. 

\begin{proof}[Proof of Theorem~\ref{thm:construction-2-phase-partial-balayage-concentration}]
First of all, let $c_{n}$ be the small positive constant $($depending only on dimension$)$ described in \cite[Theorem~7.1]{KLSS22QuadratureDomain}. Let $x \in \supp\,(\mu_{+})$ and from \eqref{eq:concentration-condition} there exists a decreasing sequence of positive numbers $\{r_{j}\}$ which converges to 0 such that 
\begin{equation*}
\mu_{+}(B_{r_{j}}(x)) > \frac{1}{c_{n}}r_{j}^{n} \quad \text{for all $j\in\mN$.}
\end{equation*}
If $\mu_{+}(\{x\})=0$, then from \eqref{eq:disjoint-condition-partial-balayage} we know that there exists $j$ such that 
\begin{equation*}
\omega_{k}\left(\left.\mu_{+}\right|_{B_{r_{j}}(x)}\right) \subset \mR^{n}\setminus\overline{\omega_{k}(\mu_{-})}. 
\end{equation*}
Applying \cite[Theorem~7.1]{KLSS22QuadratureDomain} to the measure $\left.\mu_{+}\right|_{B_{r_{j}}(x)}$ we see that  
\begin{equation}
x \in B_{r_{j}}(x) \subset \omega_{k}\left(\left.\mu_{+}\right|_{B_{r_{j}}(x)}\right) = \omega_{k,\mR^{n}\setminus\overline{\omega_{k}(\mu_{-})}}\left(\left.\mu_{+}\right|_{B_{r_{j}}(x)}\right) \subset \omega_{k,\mR^{n}\setminus\overline{\omega_{k}(\mu_{-})}}\left(\mu_{+}\right). \label{eq:verify1}
\end{equation}
If $\mu_{+}(\{x\})>0$, there exists $\epsilon>0$ such that $\epsilon \delta_{x} \le \mu_{+}$, using similar arguments as in \cite[Lemma~3.3]{GS24PartialBalayageHelmholtz}, one can show that 
\begin{equation}
x \in \omega_{k,\mR^{n}\setminus\overline{\omega_{k}(\mu_{-})}}(\epsilon \delta_{x}) \subset \omega_{k,\mR^{n}\setminus\overline{\omega_{k}(\mu_{-})}}(\mu_{+}). \label{eq:verify2}
\end{equation}
We now combine \eqref{eq:verify1} and \eqref{eq:verify2} to conclude 
\begin{equation*}
\supp\,(\mu_{+}) \subset \omega_{k,\mR^{n}\setminus\overline{\omega_{k}(\mu_{-})}}(\mu_{+}). 
\end{equation*}
Similar arguments also work for $\mu_{-}$, verifying   condition \eqref{eq:extra-support-condition}. Possibly replacing $c_{n}$ by a smaller constant, still depending only on dimension $n$, the condition \eqref{eq:k-small-condition} can be guaranteed by \eqref{eq:smallness-wavenumber-k}. Therefore we conclude Theorem~\ref{thm:construction-2-phase-partial-balayage-concentration} using Theorem~\ref{thm:construction-2-phase-partial-balayage}. 
\end{proof}

\appendix 

\section{The Pompeiu problem  (A minimization viewpoint)\label{appen:saddle-point}}

Given a ball $B_{R}$ with radius $R>0$  we consider the functional corresponding to $k_{1}=k_{2}=k$, 
$f_{1}=-1$, $f_{2}=1$ (This corresponds to Theorem~\ref{thm:local-properties-unbounded-f} with $\mu_{i}=0$ and $\lambda_1=1$, and $\lambda_2=-1$): 
\begin{equation*}
\tilde{\mJ}_{k,k}(U) = \int_{B_{R}} \left( \abs{\nabla U}^{2} - k^{2} U^{2} + 2U \right) \, \rmd x \quad \text{$U\in H_{0}^{1}(B_{R})$.}
\end{equation*}
In view of Theorem~\ref{thm:local-properties}, we consider the function $\tilde{u}$ which is a solution of the following (unstable) two-phase problem, or no-sign one-phase problem:   
\begin{equation}
(\Delta + k^{2})\tilde{u} = \chi_{D} \text{ in $\mR^{n}$} ,\quad \tilde{u}=0 \text{ outside $D$,} \label{eq:global-minimizer-2phase-1}
\end{equation}
provided $\overline{D}\subset B_{R}$. If $k < k_{*}(B_{R}) = j_{\frac{n-2}{2},1}k^{-1}$, then the (unique) minimizer of $\tilde{\mJ}_{k,k}$ must trivial. This suggests us to consider the case when $k > k_{*}(B_{R})$. In this case, the above functional is indeed unbounded below in $H_{0}^{1}(B_{R})$, see \cite[Lemma~3.1]{KSS23Minimization} or a more concrete example below in \eqref{eq:unbounded-below}. The main theme of this appendix is to exhibit some interesting points of $\tilde{\mJ}_{k,k}$ as below: 

\begin{thm}\label{thm:saddle-point}
Let $\tilde{u}$ be a solution of \eqref{eq:global-minimizer-2phase-1} for some bounded Lipschitz domain $D$ in $\mR^{n}$, then it satisfies 
\begin{equation}
\norm{\nabla\tilde{u}}_{L^{2}(D)}^{2} = \frac{n}{2k^{2}}\abs{D} ,\quad k^{2}\norm{\tilde{u}}_{L^{2}(D)}^{2} = \frac{n+2}{2k^{2}}\abs{D}, \label{eq:energy-estimate-Pompeiu-3}
\end{equation}
and it is neither a local minima nor local maxima of the functional $\tilde{\mJ}_{k,k}$ in $H_{0}^{1}(B_{R})$ for each $R>0$ with $\overline{D}\subset B_{R}$. 
\end{thm}

\begin{rem*}
See Example~\ref{exa:2-dimensional-case} below for an example of such $D$ in \eqref{eq:global-minimizer-2phase-1}. 
\end{rem*}

\begin{proof}[Proof of Theorem~\ref{thm:saddle-point}]
By integrating the identity $\nabla \cdot (x\abs{u}^{2}) = n\abs{u}^{2} + 2 u x \cdot \nabla u$ over $D$, one can easily obtain 
\begin{subequations}
\begin{equation}
(u , x\cdot\nabla u)_{L^{2}(D)} = - \frac{n}{2}\norm{u}_{L^{2}(D)}^{2} + \frac{1}{2} \br{x\cdot\nu,\abs{u}^{2}}_{\partial D} \label{eq:integral-identity1}
\end{equation}
for all real-valued $u\in H^{1}(D)$. On the other hand, combining the equations\footnote{These differential identities are suggested in \cite{CF06Helmholtz,FLL15IIPHelmholtz}.}
\begin{equation*}
\nabla\cdot(x\abs{\nabla u}^{2}) = n \abs{\nabla u}^{2} + x \cdot \nabla (\abs{\nabla u}^{2})
\end{equation*}
and 
\begin{equation*}
\begin{aligned}
& \quad x \cdot \nabla (\abs{\nabla u}^{2}) = 2 \left( \nabla\cdot(\nabla u(x\cdot\nabla u)) - \Delta u(x\cdot\nabla u)\right) - 2\abs{\nabla u}^{2} \\
& \quad = 2 \left( \nabla u\cdot\nabla(x\cdot\nabla u)\right) - 2\abs{\nabla u}^{2},
\end{aligned}
\end{equation*}
we obtain 
\begin{equation*}
\nabla u\cdot\nabla(x\cdot\nabla u) = \frac{2-n}{2} \abs{\nabla u}^{2} + \frac{1}{2} \nabla\cdot(x\abs{\nabla u}^{2}).
\end{equation*}
Then by integrating the above identity over $D$ we obtain 
\begin{equation}
(\nabla u,\nabla(x\cdot\nabla u))_{L^{2}(D)} = \frac{2-n}{2} \norm{\nabla u}_{L^{2}(D)}^{2} + \frac{1}{2}\br{x\cdot\nu , \abs{\nabla u}^{2}}_{\partial D}. \label{eq:integral-identity2}
\end{equation}
\end{subequations}
for all real-valued $u\in H^{2}(D)$. See also \cite[Lemma~2.3]{FLL15IIPHelmholtz} for a probabilistic version of \eqref{eq:integral-identity1}--\eqref{eq:integral-identity2}. 

If we choose $u=\tilde{u}$, then from \eqref{eq:integral-identity1}--\eqref{eq:integral-identity2} we see that 
\begin{equation*}
(\tilde{u},x\cdot\nabla\tilde{u})_{L^{2}(D)} = -\frac{n}{2}\norm{\tilde{u}}_{L^{2}(D)}^{2} ,\quad (\nabla \tilde{u},\nabla(x\cdot\nabla \tilde{u}))_{L^{2}(D)} = \frac{2-n}{2} \norm{\nabla\tilde{u}}_{L^{2}(D)}^{2} 
\end{equation*}
because $\tilde{u}=\abs{\nabla \tilde{u}}=0$ on $\partial D$. By testing \eqref{eq:global-minimizer-2phase-1} using $x\cdot\nabla\tilde{u}$ over $D$, we see that
\begin{equation}
\begin{aligned}
& - n \int_{D}\tilde{u} \, \rmd x = \int_{\partial D} (x\cdot\nu) \tilde{u} \, \rmd S(x) - n \int_{D}\tilde{u} \, \rmd x = \int_{D} x\cdot\nabla \tilde{u} \, \rmd x \\
& \quad = \int_{D} (\Delta\tilde{u}+k^{2}\tilde{u})(x\cdot\nabla\tilde{u}) \, \rmd x = - (\nabla \tilde{u},\nabla(x\cdot\nabla \tilde{u}))_{L^{2}(D)} + k^{2} (\tilde{u} , x\cdot\nabla \tilde{u})_{L^{2}(D)} \\
& \quad = \frac{n-2}{2}\norm{\nabla\tilde{u}}_{L^{2}(D)}^{2} - \frac{k^{2}n}{2}\norm{\tilde{u}}_{L^{2}(D)}^{2},
\end{aligned} \label{eq:Pompeiu1}
\end{equation}
On the other hand, we integrate \eqref{eq:global-minimizer-2phase-1} over $D$ to obtain 
\begin{equation}
\abs{D} = \int_{D} \Delta \tilde{u} \, \rmd x + k^{2} \int_{D} \tilde{u} \, \rmd x = \int_{\partial D} \nu \cdot \nabla \tilde{u} \, \rmd S(x) + k^{2}\int_{D}\tilde{u} \, \rmd x = k^{2}\int_{D}\tilde{u} \, \rmd x. \label{eq:Pompeiu2}
\end{equation}
We combine \eqref{eq:Pompeiu1} and \eqref{eq:Pompeiu2} to obtain the energy estimate 
\begin{equation}
-\frac{n}{k^{2}}\abs{D} = \frac{n-2}{2}\norm{\nabla\tilde{u}}_{L^{2}(D)}^{2} - \frac{k^{2}n}{2}\norm{\tilde{u}}_{L^{2}(D)}^{2}. \label{eq:energy-estimate-Pompeiu}
\end{equation}
We also test \eqref{eq:global-minimizer-2phase-1} by $\tilde{u}$, and using \eqref{eq:Pompeiu2} to obtain 
\begin{equation}
\begin{aligned}
\frac{1}{k^{2}}\abs{D} = \int_{D} \tilde{u} \, \rmd x = \int_{D}(\Delta\tilde{u}+k^{2}\tilde{u})\tilde{u} \, \rmd x = - \norm{\nabla\tilde{u}}_{L^{2}(D)}^{2} + k^{2}\norm{\tilde{u}}_{L^{2}(D)}^{2} 
\end{aligned}\label{eq:energy-estimate-Pompeiu-2}
\end{equation}
Solving \eqref{eq:energy-estimate-Pompeiu} and \eqref{eq:energy-estimate-Pompeiu-2}, we reach \eqref{eq:energy-estimate-Pompeiu-3}, that is, 
\begin{equation*}
\frac{n+2}{2}\abs{D} = k^{4}\norm{\tilde{u}}_{L^{2}(D)}^{2} ,\quad \frac{n}{2}\abs{D} = k^{2}\norm{\nabla\tilde{u}}_{L^{2}(D)}^{2}. 
\end{equation*}
For each $t\in\mR$, from \eqref{eq:Pompeiu2} and \eqref{eq:energy-estimate-Pompeiu-3} we see that 
\begin{equation*}
\tilde{\mJ}_{k,k}(t\tilde{u}) = t^{2}\left( \norm{\nabla\tilde{u}}_{L^{2}(D)}^{2} - k^{2} \norm{\tilde{u}}_{L^{2}(D)}^{2} \right) + 2t\int_{D}\tilde{u} \, \rmd x = (-t^{2}+2t)\frac{\abs{D}}{k^{2}}
\end{equation*}
for all $R>0$ with $\overline{D} \subset B_{R}$. It is clear that 
\begin{equation}
\lim_{t\rightarrow\pm\infty}\tilde{\mJ}_{k,k}(t\tilde{u}) = -\infty ,\quad \max_{t\in\mR}\tilde{\mJ}_{k,k}(t\tilde{u}) = \tilde{\mJ}_{k,k}(\tilde{u}) = \frac{\abs{D}}{k^{2}}, \label{eq:unbounded-below}
\end{equation}
which shows that $\tilde{\mJ}_{k,k}$ is unbounded below in $H_{0}^{1}(B_{R})$ and $\tilde{u}$ is not a local minimum. We now fix an eigenfunction $\phi\in H_{0}^{1}(B_{R})$ with $-\Delta\phi = k_{0}^{2} \phi$ with $k_{0} > k$. We see that 
\begin{equation*}
\begin{aligned}
&\tilde{\mJ}_{k,k}(\tilde{u}+t\phi) = \tilde{\mJ}_{k,k}(\tilde{u}) + 2t\int_{D} \nabla\tilde{u}\cdot\nabla\phi\,\rmd x + t^{2}\norm{\nabla\phi}_{L^{2}(D)}^{2} \\
& \qquad - 2k^{2}t\int_{D} \tilde{u}\phi\,\rmd x - t^{2}k^{2}\norm{\phi}_{L^{2}(D)}^{2} + 2t\int_{D}\phi\,\rmd x \\
& \quad = \tilde{\mJ}_{k,k}(\tilde{u}) + 2t(k_{0}^{2}-k^{2})\int_{D} \tilde{u}\phi\,\rmd x + t^{2}(k_{1}^{2}-k^{2})\norm{\phi}_{L^{2}(D)}^{2} + 2t\int_{D}\phi\,\rmd x \\
& \quad \ge \tilde{\mJ}_{k,k}(\tilde{u}) + 2t \left( (k_{0}^{2}-k^{2})\int_{D} \tilde{u}\phi\,\rmd x + \int_{D}\phi\,\rmd x \right).
\end{aligned}
\end{equation*}
If $(k_{0}^{2}-k^{2})\int_{D} \tilde{u}\phi\,\rmd x \ge \int_{D}\phi \, \rmd x$, then we see that 
\begin{equation*}
\tilde{\mJ}_{k,k}(\tilde{u}+t\phi) \ge \tilde{\mJ}_{k,k}(\tilde{u}) \quad \text{for all $t \ge 0$;}
\end{equation*}
Otherwise, if $(k_{0}^{2}-k^{2})\int_{D} \tilde{u}\phi\,\rmd x \le \int_{D}\phi \, \rmd x$ then we see that 
\begin{equation*}
\tilde{\mJ}_{k,k}(\tilde{u}+t\phi) \ge \tilde{\mJ}_{k,k}(\tilde{u}) \quad \text{for all $t \le 0$.}
\end{equation*}
In either case, we see that $\tilde{u}$ is not a local maximum. 
\end{proof}

\begin{exa}[A refinement of {\cite[Example~2.3]{KLSS22QuadratureDomain}}] \label{exa:2-dimensional-case}
For each $m=1,2,3,\cdots$, we now show that $B_{k^{-1}j_{\frac{n}{2},m}}$ is a null $k$-quadrature domain. We consider the function $x\mapsto \abs{x}^{\frac{2-n}{2}}J_{\frac{n-2}{2}}(k\abs{x})$ that solves $(\Delta + k^{2})u=0$ in $\mR^{n}$. By using the fact 
\begin{equation*}
(t^{\frac{2-n}{2}}J_{\frac{n-2}{2}}(t))'=-t^{\frac{2-n}{2}}J_{\frac{n}{2}}(t),
\end{equation*}
one can easily see that $\{k^{-1}j_{\frac{n}{2},m}: \text{$m$ is odd}\}$ are all positive local minima with $J_{\frac{n-2}{2}}(j_{\frac{n}{2},m})<0$, while $\{k^{-1}j_{\frac{n}{2},m}: \text{$m$ is even}\}$ are all positive local maxima with $J_{\frac{n-2}{2}}(j_{\frac{n}{2},m})>0$. For each $m\in\mN$, we define 
\begin{equation*}
\tilde{u}_{m}(x) := \left\{\begin{aligned}
& \frac{(k^{-1}j_{\frac{n}{2},m})^{\frac{2-n}{2}}J_{\frac{n-2}{2}}(j_{\frac{n}{2},m})-\abs{x}^{\frac{2-n}{2}}J_{\frac{n-2}{2}}(k\abs{x})}{k^{2}(k^{-1}j_{\frac{n}{2},m})^{\frac{2-n}{2}} J_{\frac{n-2}{2}}(j_{\frac{n}{2},m})} && \text{for all $\abs{x}<k^{-1}j_{\frac{n}{2},m}$,} \\
& 0 && \text{otherwise.}
\end{aligned}\right.
\end{equation*}
We see that $\chi_{\{\tilde{u}_{m}>0\} \cup \{\tilde{u}_{m}<0\}} = \chi_{B_{k^{-1}j_{\frac{n}{2},m}}}$ and $\tilde{u}_{m}\in C^{1,1}(\mR^{n})$ satisfy \eqref{eq:global-minimizer-2phase-1}. It is interesting to mention for the case $m=1$ that $\tilde{u}_{1} \ge 0$ in $\mR^{n}$. 
\end{exa}

Recall that the assumption in the \href{https://www.scilag.net/problem/G-180522.1}{Pompeiu problem} \cite{Pom29PompeiuProblem} is equivalent to the existence of a function $\tilde{u}$ solving the two-phase problem \eqref{eq:global-minimizer-2phase-1} for some $k > 0$, as demonstrated in \cite{Wil76PompeiuProblem, Wil81PompeiuProblem}. It's worth mentioning that \cite{Wil81PompeiuProblem} guarantees that if $D$ has a Lipschitz boundary $\partial D$ which is homeomorphic to the unit sphere in $\mR^{n}$ and it satisfies the assumption in the Pompeiu problem, then the boundary of such $D$ must be analytic. However, the unanswered question, posed in \cite[Problem~80]{Yau83OpenProblems}, is whether $D$, a bounded Lipschitz domain homeomorphic to a ball and satisfying \eqref{eq:global-minimizer-2phase-1}, must be a ball or not.

Partial results exist \cite{Aviles86Pompeiu, BK82Pompeiu, BST73Pompeiu, GS93Pompeiu} as partial answers to this question. In \cite{KLSS22QuadratureDomain}, it is observed that such a domain $D$ is also a $k$-quadrature domain. Therefore, using the maximum principle along with the positivity of the first Dirichlet eigenfunction of $-\Delta$, it is necessary that $k > k_{*}(D)$. This problem is challenging from the following perspective:

\begin{itemize}
\item By using Theorem~\ref{thm:saddle-point}, one sees that nontrivial local minima (if they exist) of the functional $\tilde{\mJ}_{k,k}$ in $H_{0}^{1}(B_{R})$ never satisfy \eqref{eq:global-minimizer-2phase-1}. We do not see how to study the symmetry of null $k$-quadrature domain by directly using the ideas in \cite[Corollary~1]{AS16MultiPhaseQD}.
\item The lack of positivity of solution to the Pompeiu problem, is also an obstacle for using the moving plane technique. 
\item It is easy to see that $k>0$ is also a Neumann eigenvalue of $D$ with eigenfunction $\tilde{v}=\tilde{u}-k^{-2}$, which satisfies $\left.\tilde{v}\right|_{\partial D}=-k^{-2}$. One also can see e.g. \cite{GN13LaplaceEigenfunctions} for isoperimetric inequality for (Dirichlet, Neumann or Robin) eigenvalues. The main difficulty is the knowledge of $\left.\tilde{v}\right|_{\partial D}$ does not explicitly contained in the Courant minimax characterization of Neumann eigenvalues. Therefore we also believe that the Courant minimax principle is not helpful in the study of the Pompeiu problem.
\end{itemize}

\section*{Declarations}

\noindent {\bf  Data availability statement:} All data needed are contained in the manuscript.

\medskip
\noindent {\bf  Funding and/or Conflicts of interests/Competing interests:} The authors declare that there are no financial, competing or conflict of interests.

\end{sloppypar}

\end{document}